\UseRawInputEncoding
\pdfoutput=1
\documentclass[11pt]{article}
\usepackage{amssymb,amsmath}
\usepackage{graphics}
\usepackage{float}
\usepackage{booktabs}
\usepackage{subfigure}
\usepackage{color,fullpage}

\usepackage{cite}
\usepackage{bm}
\usepackage[title]{appendix}
\usepackage{algorithm}%
\usepackage{algorithmicx}%
\usepackage{multirow}
\usepackage{pgfplots}
\usepackage{amsthm}
\usepackage{enumerate}
\graphicspath{{figures/}{Figures/}{logo/}}

\newtheorem{theorem}{Theorem}[section]
\newtheorem{lemma}{Lemma}[section]
\newtheorem{definition}{Definition}[section]
\newtheorem{example}{Example}[section]

\newtheorem{remark}{Remark}[section]

\makeatletter

\@addtoreset{equation}{section} 
\makeatother
\begin{document}
	
\title{Adaptive Momentum via Minimal Dual Function for Accelerating Randomized Sparse Kaczmarz}

\author{
	Lu Zhang\thanks{College of Science, National University of Defense Technology,
		Changsha, Hunan 410073, People's Republic of China. Email: \texttt{zhanglu21@nudt.edu.cn}}
\and	Jinchuan Zeng\thanks{College of Science, National University of Defense Technology,
		Changsha, Hunan 410073, People's Republic of China. Email: \texttt{zengjinchuan23@nudt.edu.cn}}
\and	Hongxia Wang\thanks{College of Science, National University of Defense Technology,
		Changsha, Hunan 410073, People's Republic of China. Email: \texttt{wanghongxia@nudt.edu.cn}}
	\and Hui Zhang\thanks{Corresponding author.
		College of Science, National University of Defense Technology,
		Changsha, Hunan 410073, People's Republic of China. Email: \texttt{h.zhang1984@163.com}
	}
}
\maketitle

\begin{abstract}
Recently, the randomized sparse Kaczmarz method has been accelerated by designing heavy ball momentum adaptively via a minimal-error principle. In this paper, we develop a new adaptive momentum method based on the minimal dual function principle to go beyond the exact measurement restriction of the minimal-error principle. Moreover, by integrating the new adaptive momentum method with the quantile-based sampling, we introduce a general algorithmic framework, called quantile-based randomized sparse Kaczmarz with minimal dual function momentum, which provides a unified approach to exact, noisy, or corrupted linear systems. In addition, we utilize the discrepancy principle and monotone error  as stopping rules for the proposed algorithm. Theoretically, we establish linear convergence in expectation of Bregman distance up to a finite horizon related to the contaminated level. At last, we provide numerical illustrations on simulated and real-world data to demonstrate the effectiveness of our proposed method.
\end{abstract}



\section{Introduction}
Recovering an unknown vector from linear measurements in the inverse problem $Ax=b$ is a foundational task across scientific computing and data science, with prominent applications in computed tomography (CT) reconstruction \cite{hansen2021computed,elfving2014semi} and machine learning \cite{bottou2018optimization,hastie2009elements}.
The prevalence of linear systems underscores the necessity for developing efficient, robust, and scalable algorithms. In this paper, 
we focus on the regularized Basis Pursuit problem
\begin{equation}
\label{eq1.2}
\min_{x\in\mathbb{R}^n} f(x):=\lambda\|x\|_1+\frac{1}{2}\|x\|^2,~s.t.~Ax=b,
\end{equation}
see e.g.\cite{cai2009convergence,yin2010analysis,lorenz2014linearized}, where $A\in\mathbb{R}^{m\times n}$ is the measurement matrix and $b\in\mathbb{R}^m$ is the measurement vector. 

The remarkable advantage of \eqref{eq1.2} over the original linear systems $Ax=b$ lies in the solution uniqueness due to the strong convexity of the objective function $f$. There are many effective methods for solving \eqref{eq1.2}. In particular, the randomized Kaczmarz (RK) method in \cite{strohmer2009randomized,zouzias2013rek,steinerberger2021sv} for $\lambda=0$ and randomized sparse Kaczmarz (RaSK) method in \cite{lorenz2014linearized,schopfer2019linear,yuan2022adaptively} for $\lambda>0$ have demonstrated their significant potential in solving (\ref{eq1.2}) by offering low computational cost and strong theoretical guarantees.
Very recently, the authors of \cite{lorenz2023minimal} proposed the Bregman-Kaczmarz with exact minimal error momentum (BK-EM) method as an accelerated variant of RaSK. Let $a_{i_k}^T$ denote the $i_k$-th row of the matrix $A$, assumed $\|a_{i_k}\|=1$ by rescaling, and let $b_{i_k}$ be the $i_k$-th element of $b$. The BK-EM method is given by
\begin{align}
\label{eq1:3}
\begin{aligned}
x_{k+1}^*=x_{k}^*-&\alpha_k(a_{i_k}^Tx_k-b_{i_k})a_{i_k}+w_k(x_k^*-x_{k-1}^*),\\
&x_{k+1}=\nabla f^*(x_{k+1}^*).
\end{aligned}
\end{align}
Here $\alpha_k$ is the relaxed parameter, $w_k$ is the momentum parameter, and $i_k$ is the index of the randomly sampled hyperplane. 
Since the efficiency of BK-EM crucially depends on parameters $\alpha_k$ and $w_k$, how to select $\alpha_k$ and $w_k$ becomes a critical issue.
Inspired by the exact-step randomized sparse Kaczmarz method in \cite{lorenz2014linearized} and the minimum error approach in \cite{zeng2024adaptiveh}, 
the authors of \cite{lorenz2023minimal} select $\alpha_k$ and $w_k$ with the least Bregman distance between $x_{k+1}$ and the unique optimal solution $\hat{x}$ with respect to $f$ and $x_{k+1}^*\in\partial f(x_{k+1})$, i.e.,
\begin{equation}
\label{eq1.4}
(\alpha_k,w_k):
=\arg\min_{(\alpha,w)} D_f^{x_{k+1}^*}(x_{k+1},\hat{x})=f^*(x_{k+1}^*)-\langle x_{k+1}^*,\hat{x}\rangle +f(\hat{x}).
\end{equation} 
However, in order to deal with the unknown $\hat{x}$, the minimal-error principle (1.3) needs to assume that the linear system is consistent and the measurement vector $b$ is exact. This requirement excludes many applications in practice.  
In order to go beyond this limit, we assume that the difference between the measurement vector $\tilde{b}$ and the exact vector $b$ can be bounded, that is, 
\begin{equation}
\label{eq1}
\|\tilde{b}-b\|_2\leq\delta,
\end{equation}
where $\delta>0$ is the contaminated level. The measurement vector $\tilde{b}$ can be written as the sum of the exact vector $b$ and a contaminated term $b^*$. The perturbation $b^{*}$ either represents small-scale noise $r$, which affects all equations but does not appreciably destroy convergence, or large-scale corruption $b^c$, which occurs in a few equations and can severely disrupt convergence.
 To handle this issue, recent quantile-based modifications of Kaczmarz method in \cite{haddock2022quantile,cheng2022block,steinerberger2023quantile,jarman2021quantilerk,zhang2024quantile,coria2024quantile,cai2025subsample,haddock2023subsampled} have been designed based on the observation that a large absolute value of residual implies a high probability of being corrupted for an equation.
As a natural development, we wonder whether the algorithmic scheme (\ref{eq1:3}) equipped with quantile statistics still works for solving corrupted systems. The main difficulty behind is how to design the parameter $\alpha_k$ and $w_k$ adaptively based on the measurement data $\tilde{b}$.

Instead of minimizing the Bregman distance between $x_{k+1}$ and $\hat{x}$ in \eqref{eq1.4}, we seek for $\alpha=\alpha_k$ and $w=w_k$ by minimizing a \textsl{perturbed} dual objective function, i.e., 
\begin{equation}
\label{eq1.7}
(\alpha_k,w_k):
=\arg\min_{(\alpha,w)} \tilde{\Psi}(y_{k+1})=f^*(A^Ty_{k+1})-\langle y_{k+1},\tilde{b}\rangle,
\end{equation}
where the dual variable $y_{k+1}$ is updated by 
$$y_{k+1}=y_k-\alpha e_{i_k}\nabla_{i_k}\tilde{\Psi}(y_k)+w(y_k-y_{k-1})$$ 
with $e_{i_k}$ being the $i_k$-th standard basis vector and $\nabla_{i_k}\tilde{\Psi}(y_k)$ the $i_k$-th partial derivative. 
It can be transferred to the primal space by 
$$x_{k+1}=\nabla f^*(A^Ty_{k+1}).$$
Interestingly, when the measurement vector is exact, the \textsl{perturbed} dual objective function reduces to the original dual objective function, which is equivalent to minimizing the Bregman distance between the iterates and the primal solution. Therefore, there is no difference between \eqref{eq1.4} and \eqref{eq1.7}. Consequently, our proposed principle \eqref{eq1.7} is more general and has the ability to deal with contaminated systems. 

The main contributions of this paper can be summarized as follows.
\begin{enumerate}[(a)]
	\item We propose a new principle to adaptively select momentum parameters based on minimizing a \textsl{perturbed} dual problem (see Section \ref{sec3.1}). The new principle overcomes the exact measurement limitation of the minimal-error principle proposed in \cite{lorenz2023minimal,zeng2024adaptiveh}.
	\item We introduce a general algorithmic framework, called Quantile-RaSK-MM, which utilizes the minimal dual function principle and the quantile-based sampling. The new framework is applicable to solving exact, noisy, or even corrupted linear systems (see Section \ref{sec3.2.1}). 
	\item We incorporate the discrepancy principle (DP) and monotone error (ME) in \cite{hansen2012air,hansen2018air} as stopping rules in Algorithm \ref{al4} to prevent semi-convergence phenomenon (see Section \ref{sec3.3}). For the ME rule, we generalize the Euclidean error measure to the Bregman distance. To our knowledge, this is the first study of stopping rules for quantile-based Kaczmarz methods. Numerical experiments show that both DP and ME are effective for noisy systems, while ME remains effective for corrupted systems (see Section \ref{sec5.3.3} and Section \ref{sec5.5}).
	\item 
	Theoretically, we show that Quantile-RaSK-MM converges linearly in expectation of Bregman distance to a finite horizon associated with the contaminated level (see Theorem \ref{th4.3}). Extensive numerical experiments validate that Quantile-RaSK-MM achieves better performance in terms of convergence speed and accuracy compared to existing methods. As a byproduct, RaSK-MM also shows good numerical performance and is comparable to BK-EM in \cite{lorenz2023minimal}.
\end{enumerate}

\subsection{The related work}
There are numerous extensions and developments of the randomized Kaczmarz and quantile-based randomized Kaczmarz methods; we review the most relevant results in this paper.

\subsubsection{Randomized Kaczmarz method and its variants}

The RK method \cite{karczmarz1937angenaherte,strohmer2009randomized,needell2010randomized} is a popular projection-based iterative method for finding the minimum $l_2$-norm solutions of full-rank and overdetermined linear systems. 
The RaSK method, a sparse variant of RK, approximates the sparse solution of linear systems without any limitations on matrices \cite{lorenz2014linearized,schopfer2019linear}. The authors of \cite{petra2015randomized} connect RaSK with the randomized coordinate descent method via duality. 
Accelerations of RK and RaSK include various sampling rules \cite{zhang2022weighted,yuan2022sparse}, averaging variants \cite{necoara2019faster,tondji2023faster} and block variants \cite{petra2015randomized,tondji2024adaptive}.  
In recent years, Polyak's heavy ball momentum \cite{polyak1964some,ghadimi2015global}, known as an excellent tool to speed up optimization algorithms, has been applied to accelerate the RK method in \cite{zeng2024adaptiveh} and the randomized extended Kaczmarz (REK) method in \cite{zeng2024adaptive}. The authors of \cite{lorenz2023minimal} propose a randomized Bregman-Kaczmarz method with minimal error momentum (BK-EM) method, and design the minimal-error principle measured by Bregman distance to adaptively choose the momentum parameter for solving exact linear systems. 
Very recently, the algorithm in \cite{jin2024convergence} also reads as (\ref{eq1:3}) and designs parameters specifically for noisy and infinite-dimensional linear inverse problems, without leveraging iterative information. 
In this work, we propose a new principle to select parameters via the least \textsl{perturbed} dual function, which utilizes iterative information and ensures computability of momentum parameters for contaminated systems.
For more papers on accelerations of iterative algorithms for solving inverse problems, we refer readers to
\cite{loizou2020momentum,jin2024adaptive}.

\subsubsection{Quantile-based RK method and its variants}

When solving corrupted linear systems, sampling a severely corrupted hyperplane may drive the iteration away from the true solution of the corresponding exact system. 
To tackle this issue, the paper \cite{haddock2022quantile} develops a variant of RK, called the quantile-based randomized Kaczmarz (Quantile-RK) method, which employs the statistics quantile of the residual vector in each iteration to detect and avoid projecting onto corrupted rows. It shows that for certain overdetermined random matrices, the Quantile-RK method converges to the solution of the uncorrupted systems with a high likelihood.
The paper \cite{steinerberger2023quantile} generalizes the theory of Quantile-RK in \cite{haddock2022quantile} to any measurement matrices.
Other works developing the Quantile-RK method include
\cite{jarman2021quantilerk,cheng2022block,coria2024quantile,battaglia2024reverse,zhang2024quantile,cai2025subsample,haddock2023subsampled}; in particular, \cite{zhang2024quantile} proposes a sparse variant, Quantile-RaSK, designed to recover the sparse solutions of exact linear systems even in the presence of large corruptions.
By combining Quantile-RaSK and heavy ball momentum, we are going to accelerate Quantile-RaSK and establish corresponding theoretical results.

\subsubsection{The stopping rules}

When applying iterative regularization methods to noisy linear systems with a known noisy level, \cite{hamarik2001monotone} uses the monotone error (ME) as a stopping rule to enforce the monotone decrease of the error. Subsequently, \cite{elfving2007stopping} employs both ME and the discrepancy principle (DP) as stopping criteria for Landweber-type iterations in linear inverse problems, unifies the two rules within a single framework, and provides a training procedure for tuning the discrepancy parameters.
The papers \cite{hansen2012air,hansen2018air} employ AIRTools to set up a parallel-beam tomography test problem, and compute reconstructions by applying the randomized Kaczmarz method, terminated by the discrepancy principle and monotone error rule.
By incorporating the spirit of the DP rule, the authors of \cite{jin2023stochastic,jin2024convergence} propose a step-size selection rule that efficiently suppresses the oscillations in iterates and mitigates the effect of semi-convergence.
Moreover, \cite{jin2024adaptive} analyzes the property of the Landweber-type method with heavy ball momentum, terminated by the discrepancy principle.

\subsection{Organization}

The rest of this paper is organized as follows. Section \ref{sec2} reviews the preliminaries of convex analysis and deduces the \textsl{perturbed} dual problem. In Section \ref{sec3}, we introduce the detailed principle of selecting parameters and stopping criteria, and we provide the formal pseudocode. Section \ref{sec4} presents a theoretical result of Quantile-RaSK-MM showing linear convergence down to a corrupted horizon.
In Section \ref{sec5}, we empirically demonstrate the effectiveness of our proposed method through experiments on simulated and real-world data. 
Finally, Section \ref{sec6} is our conclusions and discussions for future work.

\section{Notation and Preliminaries}
\label{sec2}
\subsection{Notation}
For the matrix $A\in\mathbb{R}^{m\times n}$, we use $\mathcal{R}(A)$, $a_i^T$,  $\|A\|_F$, and $\sigma_{\min}(A)$ to denote the range space, the $i$-th row, 
the Frobenius norm, and the smallest nonzero singular
value of the matrix $A$, respectively.  
Let $\beta$ be the fraction of corruptions and $q$ be the fraction of acceptable equations, and use $[\cdot]$ for the integral function; for simplicity, we assume that $\beta m$ and $qm$ are integers instead of using $[\beta m]$ and $[qm]$. Let $A_{I,J}$ be the submatrix consisting of the rows of $A$ indexed by $I$ and the columns of $A$ indexed by $J$. Define
\begin{align}
\label{2:1}
&\tilde{\sigma}_{q-\beta,\min}(A):=\min\{
\sigma_{\min}(A_{I,J})
\mid |I|=(q-\beta)m,I\subset
\{1,\dots,m\},J\subset
\{1,\dots,n\},\sigma_{\min}(A_{I,J})>0
\}.
\end{align}
For a sequence $z_i,1\leq i\leq n$, which is sorted from small to large as $z_{(i)},1\leq i\leq n$,
the $q\textrm{-quantile}$ of $z_i,1\leq i\leq n$ is defined by
\begin{eqnarray}		
z_q:= \left \{		
\begin{array}{lr}
z_{([nq]+1)},& \textrm{if $nq$~is~not~an integer},\\			
(z_{(nq)}+z_{(nq)+1})/2,& \textrm{otherwise}.		
\end{array}		
\right.		
\end{eqnarray}
For an arbitrary set $C$,
$|C|$ represents the number of elements in set $C$.
If without a declaration, then norm $\|\cdot\|$ refers to the Euclidean norm.

The linear space generated by $S\subset \mathbb{R}^n$ is defined by $\langle S\rangle$.
The orthogonal projection onto a linear space $V\subset \mathbb{R}^n$ is written as $P_V:\mathbb{R}^n\rightarrow V$. When $V=\langle \{x\}\rangle,\forall x\in\mathbb{R}^n$, we abbreviate $P_V$ to the notation $P_x$ and we have
$$P_{V}(y):=P_x(y)=\frac{\langle x ,y\rangle}{\|x\|^2}x,~\forall y\in\mathbb{R}^n.$$
In each $k$-th iterate with $k\in\mathbb{N}$, the expectation conditional to the values of the indices $i_0,\cdots,i_{k-1}$ is denoted by $\mathbb{E}_k$.
We partition the identity matrix as $I_m=(e_1,\ldots,e_m)\in\mathbb{R}^{m\times m}$ and the vector $y\in\mathbb{R}^m$ as $y_{i}=e_i^Ty,~ i=1,\cdots,m$.

\subsection{Convex analysis tools}
We recall some concepts and properties of convex functions \cite{2017Convex}, which are instrumental for our convergence analysis.
Let $f:\mathbb{R}^n\rightarrow \mathbb{R}$ be a proper convex function; its effective domain is denoted by 
$$\text{dom}(f):=\{x\in\mathbb{R}^n\mid f(x)<\infty\}.$$
The subdifferential of $f$ at $x\in \mathbb{R}^n$ is denoted by
\begin{eqnarray}
\partial f(x):=\{x^*\in \mathbb{R}^n\mid f(y)\geq f(x)+\langle x^*,y-x\rangle, \forall y\in \mathbb{R}^n\}.\nonumber
\end{eqnarray}
If the convex function $f$ is assumed to be differentiable, then $\partial f(x)=\{\nabla f(x)\}.$
We say that a proper, lower semi-continuous (l.s.c.) and convex function $f:\mathbb{R}^n\rightarrow \mathbb{R}$ is $\mu$-strongly convex if there exits $\mu>0$ such that for all $ x,y\in \mathbb{R}^n$ and $x^*\in \partial f(x)$, it holds that
\begin{eqnarray}
f(y)\geq f(x)+\langle x^*,y-x\rangle+\frac{\mu}{2}\|y-x\|^2.\nonumber
\end{eqnarray}
Then its conjugate function $f^*:\mathbb{R}^n\rightarrow \mathbb{R}$ with
$$f^*(x^*):=\sup_{x\in \mathbb{R}^n}\{\langle x^*,x\rangle-f(x)\}$$
is differentiable with a $1/\mu$-Lipschitz-continuous gradient in the sense that
\begin{eqnarray}
f^*(y)\leq f^*(x)+\langle \nabla f^*(x),y-x\rangle+\frac{1}{2\mu}\|y-x\|^2, \forall x,y\in \mathbb{R}^n.\nonumber
\end{eqnarray}
\begin{definition}
	Let $f:\mathbb{R}^n\rightarrow \mathbb{R}$ be a strongly convex function.
	The Bregman distance between $x,y \in \mathbb{R}^n$ with respect to $f$ and a subgradient $x^*\in \partial f(x)$ is defined as
	\begin{eqnarray}
	D_f^{x^*}(x,y):=f(y)-f(x)-\langle x^*,y-x\rangle=f^*(x^*)-\langle x^*,y\rangle+f(y).
	\nonumber
	\end{eqnarray}
\end{definition}

\begin{example}
		The function $f(x)=\lambda \|x\|_1+\frac{1}{2}\|x\|^2$ is 1-strongly convex with subdifferential given by
		$\partial f(x)=\{x+\lambda\cdot s\mid s_i=\textrm{sign}(x_i),x_i\neq 0\ and\ s_i\in[-1,1],x_i=0\}$ and conjugate function given by $f^*(x^*)=\frac{1}{2}\|\mathcal{S}_{\lambda}(x^*)\|^2$ with the soft shrinkage function $\mathcal{S}_{\lambda}(x)=\max \{\lvert x\rvert-\lambda,0 \}\cdot\text{sign}(x)$. For any $x^*=x+\lambda\cdot s\in\partial f(x)$, we have
		$$D_f^{x^*}(x,y)=\frac{1}{2}\|x-y\|^2+\lambda\cdot (\|y\|_1-\langle s,y\rangle). $$
		When $\lambda=0$, we have $D_f^{x^*}(x,y)=\frac{1}{2}\|x-y\|^2$.	
\end{example}

\subsection{The \textsl{perturbed} dual problem}
\label{sec2.3}
To connect the algorithmic scheme (\ref{eq1:3}) and the randomized coordinate descent method, we first derive the dual problem of (\ref{eq1.2}) with $b$ replaced by $\tilde{b}$, which is the backbone of our proposed principle; and then we introduce the property of the dual problem.
\begin{lemma}
	Consider the problem (\ref{eq1.2}) with $b$ replaced by $\tilde{b}$, where the contaminated right-hand vector $\tilde{b}=b+b^*$ satisfies $\|\tilde{b}-b\|\leq \delta$. If the contaminated level $\delta$ is unknown, the \textsl{perturbed} dual problem of (\ref{eq1.2}) is given by
\begin{equation}
\label{eq2.1}
\min_{y\in\mathbb{R}^m} \tilde{\Psi}(y):=f^*(A^Ty)-\langle y,\tilde{b}\rangle.
\end{equation}
\end{lemma}

\begin{proof}
		The primal problem (\ref{eq1.2}) with $b$ replaced by $\tilde{b}$ satisfying $\|\tilde{b}-b\|\leq \delta$ can be reformulated as
	$$
	\min_{x\in\mathbb{R}^n} f(x),~s.t.~Ax=\tilde{b}-b^*,b^*\in\Omega:=\{z\mid \|z\|\leq \delta\}.
	$$
	The associated Lagrangian function has the form of
	$$
	L(x,b^*,y)=f(x)+I_{\Omega}(b^*)-\langle y,Ax+b^*-\tilde{b}\rangle,x\in\mathbb{R}^n,b^*\in\mathbb{R}^m,y\in\mathbb{R}^m,
	$$
	where $I_{\Omega}(b^*)$ is called the indicator function of the set $\Omega$ and given by
	$$
	\begin{aligned}
	I_{\Omega}(b^*):=
	\left\{
	\begin{aligned}
	&0,& b^*\in \Omega,\\
	&+\infty,& b^*\notin \Omega.
	\end{aligned}
	\right.
	\end{aligned}
	$$

	The Lagrangian function induces the dual function
	$$
	\begin{aligned}
	\inf_{x\in\mathbb{R}^n}\inf_{b^*\in\mathbb{R}^m}L(x,b^*,y)
	&=\inf_{x\in\mathbb{R}^n}\inf_{b^*\in\mathbb{R}^m} \{f(x)+I_{\Omega}(b^*)-\langle y,Ax+b^*-\tilde{b}\rangle\}\\
	&=\inf_{x\in\mathbb{R}^n}
	\left\{
	f(x)-\langle y,Ax-\tilde{b}\rangle+\inf_{b^*\in\mathbb{R}^m}\{ I_{\Omega}(b^*)-\langle y,b^*\rangle \}
	\right\} \\
	&=\inf_{x\in\mathbb{R}^n}\left\{
	f(x)-\langle y,Ax-\tilde{b}\rangle-\sup_{b^*\in\Omega}\{ \langle y,b^*\rangle \}
	\right\}\\
	&=\inf_{x\in\mathbb{R}^n}\left\{
	f(x)-\langle y,Ax-\tilde{b}\rangle-\delta\|y\|
	\right\}\\
	&=-\sup_{x\in \mathbb{R}^n} \left\{\langle A^Ty,x\rangle -f(x)  \right\}+\langle y,\tilde{b}\rangle-\delta\|y\|\\
	&=-f^*(A^Ty)+\langle y,\tilde{b}\rangle-\delta\|y\|.
	\end{aligned}
	$$
	Thus, the corresponding dual problem is
	$$
	\min_{y\in\mathbb{R}^m} f^*(A^Ty)-\langle y,\tilde{b}\rangle+\delta\|y\|.
	$$
	When the contaminated level $\delta$ is unknown, we employ the following \textsl{perturbed} dual problem
$$
	\min_{y\in\mathbb{R}^m} \tilde{\Psi}(y):=f^*(A^Ty)-\langle y,\tilde{b}\rangle.
$$
	
	\end{proof}

According to the 1-strong convexity of $f$, the gradient of $\tilde{\Psi}$ with the following expression
$$
\nabla\tilde{\Psi}(y)=A\nabla f^{*}(A^Ty)-\tilde{b}
$$
is Lipschitz continuous with constant $L=\|A\|_2^2$ 
and coordinate-wise Lipschitz continuous with constant $L_i=\|a_i\|_2^2$ in \cite{tondji2024acceleration}.
An important result in \cite{nesterov2012efficiency} implies the following lemma.
\begin{lemma}[Descent Lemma]
	Suppose that $\tilde{\Psi}(y)$ is continuously differentiable function over $\mathbb{R}^m$ and coordinate-wise Lipschitz continuous with constant $L_i=\|a_i\|_2^2,~i=1,\cdots,m$. Then for any $t\in\mathbb{R}$ and $y\in\mathbb{R}^m$ we have
	$$
\tilde{\Psi}(y+te_i)
\leq 
\tilde{\Psi}(y)+ \nabla_i \tilde{\Psi} (y)\cdot t +\frac{L_i}{2}\cdot t^2,~i=1,\cdots,m.
$$
\end{lemma}
In particular, when the right-hand side $b$ is measured exactly, the dual problem of (\ref{eq1.2}) is
$$
\hat{\Psi}:=\min_{y\in\mathbb{R}^m} \Psi(y)=f^*(A^Ty)-\langle y,b\rangle.
$$
The following lemma connects the distance to the optimality of the dual problem and the Bregman distance
to the optimal
primal solution, which can also be referred to as Theorem 3.8 in \cite{jinconvergence}. 
\begin{lemma}[Lemma 3.1, \cite{tondji2024acceleration}]
	\label{lemma3:1}
	Let $b\in\mathcal{R}(A)$ and $(y_k,x_k^*,x_k)$ be three sequences such that $x_k^*=A^Ty_k,~x_k=\nabla f^*(x_k^*)$. Then it holds that
	$$
	D_f^{x_k^*}(x_k,\hat{x})=\Psi(y_k)-\hat{\Psi}.
	$$
\end{lemma}

\section{The proposed method}

\label{sec3}

In this section, we aim to solve the system $Ax=b$ with access only to the corrupted right-hand side
$$
\tilde{b}=b+b^c,
$$
where $b^c$ is the sparse vector of corruptions and $b=A\hat{x}$ is the exact data. The fraction of being corrupted is $\beta\in (0,1)$, i.e., $\|b^c\|_0=\beta m$, and it is bounded by $\|b^c\|\leq \delta$.
We begin with introducing the quantile-based sampling and the minimal dual function principle. 
Equipped with these two techniques, we give a formal description of Quantile-RaSK-MM designed for solving corrupted linear systems. In particular, Quantile-RaSK-MM can be reduced to solve noisy or exact systems, abbreviated by RaSK-MM. Finally, we extend the DP and ME rule to terminate the Quantile-RaSK-MM method.

\subsection{The technical details}
\label{sec3.1}
Applying the randomized coordinate descent method to the \textsl{perturbed} dual problem (\ref{eq2.1}) and adding heavy ball momentum to the dual update, we consider the following iterative scheme
\begin{equation}
\label{eq4.1}
\begin{aligned}
y_{k+1}=y_k-&\alpha_k e_{i_k}\nabla_{i_k}\tilde{\Psi}(y_k)+w_k(y_k-y_{k-1}),\\
x_{k+1} &= \nabla f^*(A^Ty_{k+1}).
\end{aligned}
\end{equation}
For exact linear systems, (\ref{eq4.1}) and (\ref{eq1:3}) are equivalent.
The performance of (\ref{eq4.1}) extremely depends on the sampling index $i_k$ and the parameters $\alpha_k,w_k$. Therefore, before formally proposing the algorithm, we first introduce the techniques for selecting $i_k$ and $\alpha_k,w_k$.

\subsubsection{The quantile statistics to detect corrupted equations}

Based on the observation that a larger residual indicates a higher probability of being corrupted, the absolute residual $|a_{i}^Tx_k-\tilde{b}_{i}|,i=1,\cdots,m$ is widely used to measure the error of each equation in \cite{haddock2022quantile,cheng2022block,steinerberger2023quantile,jarman2021quantilerk,battaglia2024reverse,zhang2024quantile}. 
In detail, firstly fix the acceptable ratio $q\in (\beta,1-\beta]$. In $k$-th iterate, define the $q$-quantile of absolute residuals indexed by the set $N_1=\{1,\cdots,m\}$ as
\begin{equation}
Q_q(x_k,N_1):={q\textrm{-quantile}}(\lvert\langle a_i,x\rangle-\tilde{b}_i\rvert\mid i\in N_1).\nonumber
\end{equation}
A hyperplane is regarded as acceptable if its absolute residual is less than the computed $q$-quantile. Then the acceptable set is given by
$$N_2:=\{i\in N_1\mid \lvert\langle x_k,a_{i}\rangle -\tilde{b}_{i}\rvert\leq Q_q(x_k,N_1)\}.$$
There are $|N_2|=qm$ acceptable hyperplanes. Then we sample an index $i_k$ uniformly from the acceptable set $N_2$ and use that for a step of the method (\ref{eq4.1}).    

\subsubsection{The new principle for adaptively updating parameters $\alpha_k,w_k$}



Recall that the authors of \cite{lorenz2023minimal} select parameters $\alpha_k$ and $w_k$ with minimal error measured in Bregman distance at $k$-th iterate, i.e.,
\begin{equation}
\label{eq3.5}
(\alpha_k,w_k):
=\arg\min_{(\alpha,w)} D_f^{x_{k+1}^*}(x_{k+1},\hat{x}).
\end{equation}
The principle (\ref{eq3.5}) minimizes Bregman distance from the iterate $x_k$ to the true solution $\hat{x}$, and consequently, the expression of parameter $\alpha_k$ and $w_k$ will inevitably associate with the true solution. For exact linear systems, as discussed in \cite{lorenz2023minimal}, the problem of the $\hat{x}$-dependence can be resolved by introducing a new variable and computing it in a recursive manner. However, for contaminated linear systems, it cannot be computed without the exact vector $b$.


For exact linear systems, it follows from Lemma \ref{lemma3:1} that we have 
$$
D_f^{x_{k+1}^*}(x_{k+1},\hat{x})=\Psi(y_{k+1})-\hat{\Psi},
$$
which implies that the principle (\ref{eq3.5}) is equivalent to
\begin{equation}
\label{eq3:3}
(\alpha_k,w_k):
=\arg\min_{(\alpha,w)}\Psi(y_{k+1}).
\end{equation}
Since the principle (\ref{eq3:3}) does not involve any unknown quantity, neither do parameters derived from (\ref{eq3:3}). 
In contaminated case, (\ref{eq3:3}) can be generalized to minimize the \textsl{perturbed} dual function, i.e.,
\begin{equation}
\begin{aligned}
\label{eq3:5}
(\alpha_k,w_k):
&=\arg\min_{(\alpha,w)} \tilde{\Psi}(y_{k+1})=f^*(A^Ty_{k+1})-\langle y_{k+1},\tilde{b}\rangle.
\end{aligned}
\end{equation}
We call this idea the minimal dual function principle, which still can update parameters in the absence of the exact vector $b$. 

\subsection{The proposed method}
\label{sec3.2.1}
Up to now, we have introduced techniques for selecting $i_k$ and $\alpha_k,w_k$; we are ready to derive the mathematical formulation of $\alpha_k$ and $w_k$, which needs to ensure convergence of the update (\ref{eq4.1}) in the corrupted case. Let $v_k^*=y_k-y_{k-1}$.
It follows from the Lipschitz gradient continuity of $\tilde{\Psi}$ with constant $L=\|A\|_2^2$ that we have
\begin{equation}
\label{eq5:5}
\begin{aligned}
\tilde{\Psi}(y_{k+1})
&= \tilde{\Psi}(y_k-\alpha e_{i_k}\nabla_{i_k}\tilde{\Psi}(y_k)+wv_k^*)\\
&\leq \tilde{\Psi}(y_k)+\langle -\alpha e_{i_k}\nabla_{i_k}\tilde{\Psi}(y_k)+wv_k^*, \nabla \tilde{\Psi}(y_k)\rangle +\frac{\|A\|_2^2}{2}\|-\alpha e_{i_k}\nabla_{i_k}\tilde{\Psi}(y_k)+wv_k^*\|^2.
\end{aligned}
\end{equation}
Thanks to $\tilde{\Psi}(y_{k})=\Psi(y_k)+\langle b-\tilde{b},y_k\rangle$, then (\ref{eq5:5}) can be reformulated as
\begin{equation}
\label{eq5:6}
\begin{aligned}
&~~~~
\Psi
(y_{k+1})\\
&\leq \Psi(y_k)-\langle b-\tilde{b},v_{k+1}^*\rangle+\langle -\alpha e_{i_k}\nabla_{i_k}\tilde{\Psi}(y_k)+wv_k^*, \nabla \tilde{\Psi}(y_k)\rangle +\frac{\|A\|_2^2}{2}\|-\alpha e_{i_k}\nabla_{i_k}\tilde{\Psi}(y_k)+wv_k^*\|^2\\
&\leq 
\Psi(y_k)+\frac{1}{4\gamma}\delta^2+\gamma\|y_k-y_{k+1}\|^2+\langle -\alpha e_{i_k}\nabla_{i_k}\tilde{\Psi}(y_k)+wv_k^*, \nabla \tilde{\Psi}(y_k)\rangle
+\frac{\|A\|_2^2}{2}\|-\alpha e_{i_k}\nabla_{i_k}\tilde{\Psi}(y_k)+wv_k^*\|^2\\
&=
\Psi(y_k)+\frac{1}{4\gamma}\delta^2+\langle -\alpha e_{i_k}\nabla_{i_k}\tilde{\Psi}(y_k)+wv_k^*, \nabla \tilde{\Psi}(y_k)\rangle +\left(\frac{\|A\|_2^2}{2}+\gamma\right)\|-\alpha e_{i_k}\nabla_{i_k}\tilde{\Psi}(y_k)+wv_k^*\|^2,
\end{aligned}
\end{equation}
where the second inequality holds with $\gamma>0$.

Taking derivative of (\ref{eq5:6}) with respect to $\alpha$ deduces
\begin{equation*}
\label{eq5:7}
(2\gamma+\|A\|_2^2)\|\nabla_{i_k}\tilde{\Psi}(y_k)\|^2\alpha-(2\gamma+\|A\|_2^2)\langle e_{i_k}\nabla_{i_k}\tilde{\Psi}(y_k),v_k^*\rangle w=\|\nabla_{i_k}\tilde{\Psi}(y_k)\|^2.
\end{equation*}
Taking derivative of (\ref{eq5:6}) with respect to $w$ deduces
\begin{equation}
\label{eq5:8}
-(2\gamma+\|A\|_2^2)\langle e_{i_k}\nabla_{i_k}\tilde{\Psi}(y_k),v_k^*\rangle\alpha+(2\gamma+\|A\|_2^2)\|v_k^*\|^2w=-\langle v_k^*, \nabla \tilde{\Psi}(y_k)\rangle.
\end{equation}
We can obtain the following linear system with respect to $\alpha$ and $w$
$$
\begin{aligned}
\left\{
\begin{aligned}
&(2\gamma+\|A\|_2^2)\|\nabla_{i_k}\tilde{\Psi}(y_k)\|^2\alpha-(2\gamma+\|A\|_2^2)\langle e_{i_k}\nabla_{i_k}\tilde{\Psi}(y_k),v_k^*\rangle w=\|\nabla_{i_k}\tilde{\Psi}(y_k)\|^2,\\
&-(2\gamma+\|A\|_2^2)\langle e_{i_k}\nabla_{i_k}\tilde{\Psi}(y_k),v_k^*\rangle\alpha+(2\gamma+\|A\|_2^2)\|v_k^*\|^2w=-\langle v_k^*, \nabla \tilde{\Psi}(y_k)\rangle.
\end{aligned}
\right.
\end{aligned}
$$
The linear system is invertible if
$$
\|\nabla_{i_k} \tilde{\Psi}(y_k)\|^2 \|v_k^*\|^2-\langle e_{i_k}\nabla_{i_k} \tilde{\Psi}(y_k),v_k^*\rangle^2\neq0,$$
which is equivalent to the condition that the search direction
$e_{i_k}\nabla_{i_k} \tilde{\Psi}(y_k)$ and the momentum correction $v_k^*$ are linearly independent. 
In the case of linear independence, the optimal solution $(\alpha_k,w_k)$ is given by
\begin{align}
\label{eq4.9}
\left\{
\begin{aligned}
&\alpha_k=\frac{\|\nabla_{i_k} \tilde{\Psi}(y_k)\|^2 \|v_k^*\|^2-\langle e_{i_k}\nabla_{i_k} \tilde{\Psi}(y_k),v_k^*\rangle \langle \nabla \tilde{\Psi}(y_k), v_k^*\rangle}{(2\gamma+\|A\|_2^2)(\|\nabla_{i_k} \tilde{\Psi}(y_k)\|^2 \|v_k^*\|^2-\langle e_{i_k}\nabla_{i_k} \tilde{\Psi}(y_k),v_k^*\rangle^2)},\\
&w_k=\frac{-\|\nabla_{i_k} \tilde{\Psi}(y_k)\|^2\langle \nabla \tilde{\Psi}(y_k), v_k^*\rangle+\|\nabla_{i_k} \tilde{\Psi}(y_k)\|^2\langle e_{i_k}\nabla_{i_k} \tilde{\Psi}(y_k),v_k^*\rangle }{(2\gamma+\|A\|_2^2)(\|\nabla_{i_k} \tilde{\Psi}(y_k)\|^2 \|v_k^*\|^2-\langle e_{i_k}\nabla_{i_k} \tilde{\Psi}(y_k),v_k^*\rangle^2)}.
\end{aligned}
\right.
\end{align}
In the case of linear dependence, the direction of gradient descent is the same as the direction of momentum correction. We can set $\alpha_k=1$ and using (\ref{eq5:8}) obtain 
\begin{equation}\label{eq5:3}w_k=\frac{\langle (2\gamma+\|A\|_2^2) e_{i_k}\nabla_{i_k}\tilde{\Psi}(y_k)-\nabla \tilde{\Psi}(y_k), v_k^*\rangle}{(2\gamma+\|A\|_2^2)\|v_k^*\|^2}.
\end{equation}
We conclude the iterates in Algorithm \ref{al4}.

\begin{algorithm}[H]
	\caption{Quantile-RaSK with minimal function momentum (Quantile-RaSK-MM)}\label{al4}
	\begin{algorithmic}[1]
		\State \textbf{Input}: Given $A \in \mathbb{R}^{m \times n}$, $\tilde{b} \in \mathbb{R}^{m}$,  $x_{0}=0\in \mathbb{R}^{n}$, $y_{-1}=y_0=0\in\mathbb{R}^{m}$, and parameters $\gamma,q,N$
		\State \textbf{Ouput}: solution of $\min_{x\in \mathbb{R}^n}f(x)=\lambda \|x\|_1+\frac{1}{2}\|x\|^2~s.t.~Ax=b$
		\State normalize $A$ by row
		\State \textbf{for} $k=0,1,2, \ldots,N$
		\State $\quad$$\quad$compute $N_1=\{|\langle x_k,a_{i}\rangle-\tilde{b}_{i}|\mid i\in \{1,\cdots,m\}\}$
		\State $\quad$$\quad$compute the $q$-quantile of $N_1$:  $Q_k=Q_q(x_k,N_1)$
		\State $\quad$$\quad$consider $N_2=\{i\in \{1,\cdots,m\}\mid |\langle x_k,a_{i}\rangle -\tilde{b}_{i}|\leq Q_k\}$
		\State $\quad$$\quad$sample ${i_k}$ at random with probability $ p_i=\frac{1}{qm}$, $i\in N_2$
		\State $\quad$$\quad$
		$
		v_k^*=y_k-y_{k-1}~
		$
		\State$\quad$$\quad$\textbf{If}
		$
		(a_{i_k}^Tx_k-\tilde{b}_{i_k})^2(\|v_k^*\|^2-\langle e_{i_k},v_k^*\rangle^2)>0$
		\State$\quad$$\quad$$\quad$$\quad$ $\alpha_k=\frac{(a_{i_k}^Tx_k-\tilde{b}_{i_k}) \|v_k^*\|^2-\langle e_{i_k},v_k^*\rangle \langle Ax_k-\tilde{b}, v_k^*\rangle}{(2\gamma+\|A\|_2^2)((a_{i_k}^Tx_k-\tilde{b}_{i_k}) \|v_k^*\|^2-(a_{i_k}^Tx_k-\tilde{b}_{i_k})\langle e_{i_k},v_k^*\rangle^2)}$
		\State$\quad$$\quad$$\quad$$\quad$
		$w_k=\frac{-\langle Ax_k-\tilde{b}, v_k^*\rangle+(a_{i_k}^Tx_k-\tilde{b}_{i_k})\langle e_{i_k},v_k^*\rangle }{(2\gamma+\|A\|_2^2) (\|v_k^*\|^2-\langle e_{i_k},v_k^*\rangle^2)}$
		\State$\quad$$\quad$\textbf{else}
	\State$\quad$$\quad$$\quad$$\quad$
	$\alpha_k=1$
	\State$\quad$$\quad$$\quad$$\quad$\textbf{If}
	$\|v_k^*\|>0$
	\State$\quad$$\quad$$\quad$$\quad$$\quad$$\quad$
	$w_k=\frac{\langle (2\gamma+\|A\|_2^2) e_{i_k}(a_{i_k}^Tx_k-\tilde{b}_{i_k})-(Ax_k-\tilde{b}), v_k^*\rangle}{(2\gamma+\|A\|_2^2)\|v_k^*\|^2}$
	\State$\quad$$\quad$$\quad$$\quad$\textbf{else}
	\State$\quad$$\quad$$\quad$$\quad$$\quad$$\quad$ $w_k=0$
	\State$\quad$$\quad$$\quad$$\quad$\textbf{end}
		\State$\quad$$\quad$\textbf{end}
		\State$\quad$$\quad$
		$y_{k+1}=y_k-\alpha_k e_{i_k}(a_{i_k}^Tx_k-\tilde{b}_{i_k})+w_k v_k^*$
		\State $\quad$$\quad$
		$x_{k+1}^*=A^Ty_{k+1}$	
		\State$\quad$$\quad$
		$x_{k+1}=\mathcal{S}_{\lambda}(x_{k+1}^*)$
		\State $\quad$$\quad$increment $k=k+1$
		\State \textbf{until} a stopping criterion is satisfied
	\end{algorithmic}
\end{algorithm}

The computational costs for variants of the quantile-based sparse Kaczmarz method are shown in Table \ref{table5}.
Although Quantile-RaSK-MM seems to be computationally expensive at first glance, it only involves vector-vector multiplications and matrix-vector multiplications.
In particular, when updating parameters $\alpha_k,w_k$, some quantities can be computed once and reused, e.g.,
\begin{equation}
\label{eq5:2}
s_1 = Ax_k-\tilde{b},~
s_{1,i_k} = a_{i_k}^Tx_k-\tilde{b}_{i_k},~
s_2 = \|v_k^*\|^2,~
s_3 = \langle e_{i_k},v_k^*\rangle,~
s_4 = \langle s_1,v_k^*\rangle,
\end{equation}
where $s_1$ and $s_{1,i_k}$ has been computed in the 5-th step in Algorithm \ref{al4}.
Then if $s_{1,i_k}^2(s_2-s_3^2)>0$ we have
$$
\alpha_k=\frac{s_{1,i_k}s_2-s_3s_4}{(2\gamma+\|A\|_2^2)s_{1,i_k}(s_2-s_3^2)},w_k=\frac{-s_4+s_{1,i_k}s_3}{(2\gamma+\|A\|_2^2)(s_2-s_3^2)}.
$$ 
By precomputing these values, the computational cost can be significantly reduced.
The time complexity of Quantile-RaSK-MM is $\mathcal{O}(mn)$, which is at the same level as the Quantile-RaSK in \cite{zhang2024quantile}. In contrast, the Quantile-ERaSK in \cite{zhang2024quantile} has a higher time complexity since the computation of its exact stepsize requires $\mathcal{O}(n\log (n))$-sorting procedure.

\begin{table}[H]
	\begin{center}
		\begin{minipage}{\textwidth}
			\caption{Summary of the computational costs for variants of quantile-based sparse Kaczmarz method, where $m,n$ are the numbers of rows and columns of the measurement matrix}
			\label{table5}
			\begin{tabular*}{\textwidth}{@{\extracolsep{\fill}}lcc@{\extracolsep{\fill}}}
				\toprule
				Method  &  Computational cost & Rate bound are shown in
				\\	
				\midrule
				Quantile-RaSK  & $2mn+3m+7n$    &  Theorem 2 \cite{zhang2024quantile} 
				\\
				Quantile-ERaSK  & $2mn+3m+19n+\mathcal{ O}(nlog (n))$  &  Theorem 2 \cite{zhang2024quantile}
				\\
			Quantile-RaSK-MM  & $4mn+10m+4n$  &  Section \ref{sect4.1} in this paper
				\\
				\bottomrule
			\end{tabular*}
		\end{minipage}
	\end{center}
\end{table}

When the linear system is uncorrupted, i.e., $\beta=0$, we let $q=1$ and then Quantile-RaSK-MM reduces to RaSK with minimal function momentum (RaSK-MM) method.
In the RaSK-MM method, the index $i_k$ of hyperplane is uniformly sampled from the set of all indices $\{1,\cdots,m\}$ at random; the parameters $\alpha_k,w_k$ are still computed by the mathematical formula (\ref{eq4.9}) or (\ref{eq5:3}).
The Quantile-RaSK-MM method can be viewed as a framework for solving linear systems
$Ax=b$ when only damaged observations $\tilde{b}=b+b^*$ with $\|b^{*}\|\leq \delta$ are available. This framework can be divided into three cases.
	\begin{enumerate}[(a)]
		\item For corrupted system, i.e., $b^*=b^c$ with $\|b^c\|_0=\beta m$, Quantile-RaSK-MM is employed with an appropriate acceptable ratio $q$ and $\gamma>0$. 
		\item For noisy systems, i.e., $b^*=r$, RaSK-MM with $\gamma>0$ is employed to account for the additive noise. 
		\item For exact systems, i.e., $b^*=0$, RaSK-MM with $\gamma=0$ is applied to solve the systems exactly.
	\end{enumerate}

\begin{remark}
\label{remark1}
In the Quantile-RaSK-MM method, the condition imposed on $q$ and $\beta$ is strict in the sense that $q$ must be larger than $\beta$, and we must take $q\leq 1-\beta$ to avoid sampling corrupted rows. We conclude that the acceptable ratio should satisfy $q\in(\beta,1-\beta]$ in Quantile-RaSK-MM. In contrast, Theorem 2 in \cite{zhang2024quantile} indicates that the acceptable ratio should satisfy $q\in(\beta,1-\beta)$ in Quantile-RaSK and Quantile-ERaSK; that is, different from Quantile-RaSK-MM, the parameter $q$ in Quantile-RaSK cannot reach the upper bound $1-\beta$.
\end{remark}

\subsection{The stopping criteria}
\label{sec3.3}

An appropriate stopping criterion plays a crucial role in the practical performance of Algorithm \ref{al4}. The most common choice is to terminate after a prescribed maximum number of iterations; see \cite{lorenz2023minimal,zeng2024adaptive,zeng2024adaptiveh,yuan2022adaptively,tondji2023faster,tondji2024acceleration} and the references therein. However, excessive iteration increases computational cost and may cause oscillations in iterates as presented in \cite{hansen2012air}, while insufficient iterations mean loss of accuracy in $x_k$. These considerations motivate the study of an approximate stopping criterion for Algorithm \ref{al4}.

Firstly, we directly adopt the discrepancy principle in the sense that choosing $k_{DP}$ as the first iteration index $k$ satisfying
\begin{equation}
\label{eq3.1}
\|Ax_k-\tilde{b}\|\leq \tau\delta,
\end{equation}
which has been widely used in \cite{elfving2007stopping,elfving2014semi,hansen2012air,jin2024adaptive}. Here $\delta>0$ denotes the known contaminated level and the parameter $\tau\geq1$ is a tuning parameter that can be estimated via a training process in \cite{elfving2007stopping}. When the measurement vector $\tilde{b}$ is corrupted, the residual error $\|Ax_k-\tilde{b}\|$ becomes excessively large, rendering the DP unreliable. This calls for a stopping rule that remains effective in the presence of corruptions.

Building on the monotone error  rule in \cite{hamarik2001monotone,elfving2007stopping,elfving2014semi,hansen2012air}, we define $k_{ME}$ as the largest index such that the error $D_f^{x_{k}^*}(x_{k},\hat{x})$ for $k\in [1,k_{ME-1}]$ is monotonically decrease in the sense that we have 
\begin{equation}
\label{eq3:1}
D_f^{x_{k+1}^*}(x_{k+1},\hat{x})<D_f^{x_{k}^*}(x_{k},\hat{x}),~k=1,\cdots,k_{ME-1}.
\end{equation}
Since the truth solution $\hat{x}$ is unavailable in practice, we next deduce a sufficient condition to make (\ref{eq3:1}) computable.
Based on Lemma \ref{lemma3:1} and (\ref{eq5:6}), we obtain  
\begin{equation}
\begin{aligned}
\label{eq5:1}
&~~~~
D_f^{x_{k+1}^*}(x_{k+1},\hat{x})\\
&\leq D_f^{x_{k}^*}(x_{k},\hat{x})-\langle b-\tilde{b},v_{k+1}^*\rangle+\langle -\alpha e_{i_k}\nabla_{i_k}\tilde{\Psi}(y_k)+wv_k^*, \nabla \tilde{\Psi}(y_k)\rangle +\frac{\|A\|_2^2}{2}\|-\alpha e_{i_k}\nabla_{i_k}\tilde{\Psi}(y_k)+wv_k^*\|^2\\
&\leq 
D_f^{x_{k}^*}(x_{k},\hat{x})-\langle b-\tilde{b},v_{k+1}^*\rangle+\langle -\alpha_k e_{i_k}\nabla_{i_k}\tilde{\Psi}(y_k)+w_kv_k^*, \nabla \tilde{\Psi}(y_k)\rangle +\frac{\|A\|_2^2}{2}\|-\alpha_k e_{i_k}\nabla_{i_k}\tilde{\Psi}(y_k)+w_kv_k^*\|^2,
\end{aligned}
\end{equation}
where $\alpha_k$ and $w_k$ are given by (\ref{eq4.9}).
We introduce a parameter $\tau_k\in [-1,1]$ via the Cauchy–Schwarz inequality, then with the known contaminated level $\delta$ we have
$$
-\langle b-\tilde{b},v_{k+1}^*\rangle=\tau_k\delta\|v_{k+1}^*\|.
$$
Thus, (\ref{eq5:1}) can be reformulated as
\begin{equation*}
\begin{aligned}
&~~~~
D_f^{x_{k+1}^*}(x_{k+1},\hat{x})\\
&\leq D_f^{x_{k}^*}(x_{k},\hat{x})+\tau_k\delta\|v_{k+1}^*\|+\langle -\alpha_k e_{i_k}\nabla_{i_k}\tilde{\Psi}(y_k)+w_kv_k^*, \nabla \tilde{\Psi}(y_k)\rangle +\frac{\|A\|_2^2}{2}\|-\alpha e_{i_k}\nabla_{i_k}\tilde{\Psi}(y_k)+w_kv_k^*\|^2\\
&=
D_f^{x_{k}^*}(x_{k},\hat{x})+\tau_k\delta\|v_{k+1}^*\|
+\frac{\|A\|_2^2}{2}w_k^2\|v_k^*\|^2
+w_k\langle \nabla \tilde{\Psi}(y_k)-\alpha\|A\|_2^2 e_{i_k}\nabla_{i_k}\tilde{\Psi}(y_k),v_k^*\rangle\\
&~~~~~~~~~~~~~~~~~~~~
-\left(\alpha_k-\frac{\|A\|_2^2}{2}\alpha_k^2\right)\|\nabla_{i_k}\tilde{\Psi}(y_k)\|^2.
\end{aligned}
\end{equation*}
Let $\tau=\max_k |\tau_k|$ and $\tau\in [0,1]$. Hence, we have
\begin{equation*}
 \begin{aligned}
D_f^{x_{k+1}^*}(x_{k+1},\hat{x})
 &\leq
 D_f^{x_{k}^*}(x_{k},\hat{x})+\tau\delta\|v_{k+1}^*\|
 +\frac{\|A\|_2^2}{2}w_k^2\|v_k^*\|^2
 +w_k\langle \nabla \tilde{\Psi}(y_k)-\alpha_k\|A\|_2^2 e_{i_k}\nabla_{i_k}\tilde{\Psi}(y_k),v_k^*\rangle\\
 &~~~~~~~~~~~~~~~~~~~~
 -\left(\alpha_k-\frac{\|A\|_2^2}{2}\alpha_k^2\right)\|\nabla_{i_k}\tilde{\Psi}(y_k)\|^2.
 \end{aligned}
 \end{equation*}
Therefore, the error measured by Bregman distance monotonically decreases provided that
$$
\tau\delta\|v_{k+1}^*\|
+\frac{\|A\|_2^2}{2}w_k^2\|v_k^*\|^2
+w_k\langle \nabla \tilde{\Psi}(y_k)-\alpha_k\|A\|_2^2 e_{i_k}\nabla_{i_k}\tilde{\Psi}(y_k),v_k^*\rangle
-\left(\alpha_k-\frac{\|A\|_2^2}{2}\alpha_k^2\right)\|\nabla_{i_k}\tilde{\Psi}(y_k)\|^2\leq 0.
$$
We then define the stopping index as the first index $k$ satisfying
\begin{equation}
\label{eq5:4}
\tau\delta\|v_{k+1}^*\|
\geq 
-\frac{\|A\|_2^2}{2}w_k^2\|v_k^*\|^2
-w_k\langle \nabla \tilde{\Psi}(y_k)-\alpha_k\|A\|_2^2 e_{i_k}\nabla_{i_k}\tilde{\Psi}(y_k),v_k^*\rangle
+\left(\alpha_k-\frac{\|A\|_2^2}{2}\alpha_k^2\right)\|\nabla_{i_k}\tilde{\Psi}(y_k)\|^2.
\end{equation}
With the simple notations in (\ref{eq5:2}), the stopping rule can be simplified as
\begin{equation*}
\tau\delta\|v_{k+1}^*\|
\geq 
-\frac{\|A\|_2^2}{2}w_k^2s_2-w_k(s_4-\alpha_k \|A\|_2^2 s_{1,i_k}s_3)+\left(\alpha_k-\frac{\|A\|_2^2}{2}\alpha_k^2\right)s_{1,i_k}:=S_{k+1}.
\end{equation*}

With DP (\ref{eq3.1}) and ME (\ref{eq5:4}) in place, we will evaluate their effectiveness through numerical experiments in Section \ref{sec5}.

\begin{remark}
\begin{enumerate}[(a)]
\item Since the vector $v_{k+1}^*$ is required in the next iterate, both DP (\ref{eq3.1}) and ME (\ref{eq5:4}) do not introduce additional computing cost.
\item As discussed in \cite{elfving2007stopping}, the parameter $\tau$ can be determined via a training procedure. In detail, given a test problem with known $\delta$, we separately compute the iteration number $k_{DP}$ and $k_{ME}$ that minimizes the error $D_f^{x_k^*}(x_k,\hat{x})$ and set
$$\tau_{DP}=\frac{\|Ax_{k_{DP}}-\tilde{b}\|+\|Ax_{k_{DP}-1}-\tilde{b}\|}{2\delta},~\tau_{ME}=\frac{S_{k_{ME}}/\|v_{k_{ME}}^*\|+S_{k_{ME}-1}/\|v_{k_{ME}-1}^*\|}{2\delta}.$$
\end{enumerate}
\end{remark}

\section{Convergence analysis}
\label{sec4}

In this section, we establish convergence results for the Quantile-RaSK-MM method, which are stated in a unified framework that includes Quantile-RaSK-MM for corrupted systems, RaSK-MM with $\gamma>0$ for noisy systems, and RaSK-MM with $\gamma=0$ for exact systems.

\subsection{The convergence analysis of Quantile-RaSK-MM}
\label{sect4.1}

Here we investigate the convergence result of Quantile-RaSK-MM. To this end, the Bregman distance between the iterates and the solution $\hat{x}$ can be estimated by the following error bound, which is an important lemma for the convergence analysis in this paper.

\begin{lemma}[Lemma 3.1, \cite{schopfer2019linear}]
	\label{lemma3:2}
	Let $\hat{x}$ be the unique solution of (\ref{eq1.2}).
	Denote $A_{:,J}$ be the submatrix formed by the columns of $A$ indexed by $J$, and define a minimum singular value 
	$$\tilde{\sigma}_{\min }(A):=\min \left\{\sigma_{\min }\left(A_{:,J}\right) \mid J \subset\{1, \ldots, n\}, \sigma_{\min }\left(A_{:,J}\right)>0\right\}.$$ 
	Let $ \textrm{supp}(\hat{x})=\left\{j \in\{1, \ldots, n\} \mid \hat{x}_j \neq 0\right\}$.
	When $b \neq 0$ we also have $\hat{x} \neq 0$, and hence
	$
	|\hat{x}|_{\min }:=\min \left\{\left|\hat{x}_j\right| \mid j \in \textrm{supp}(\hat{x})\right\}>0
	$. 
	Then for all $x\in\mathbb{R}^n$ with $\partial f(x)\cap \mathcal{R}(A^T)\neq \emptyset$ and all $x^*=A^Ty\in \partial f(x)\cap \mathcal{R}(A^T)$ we have
	$$
	D_f^{x^*}(x,\hat{x})\leq\frac{1}{\tilde{\sigma}^2_{\min }(A)}\cdot\frac{|\hat{x}|_{\min }+2\lambda}{|\hat{x}|_{\min }}\cdot\|Ax-b\|_2^2.
	$$
\end{lemma}

With Lemma \ref{lemma3:2}, we recall the linear convergence for RaSK with respect to Bregman distance.

\begin{theorem}[Theorem 3.2, \cite{schopfer2019linear}; Theorem 4.8, \cite{tondji2024acceleration}]
	\label{th4:2}
	Let $x_k,x_k^*$
	be the sequences generated by the RaSK method with uniform sampling and the matrix $A\in\mathbb{R}^{m\times n}$ has unit-norm rows. Then it converges in expectation to the unique
	solution $\hat{x}$ of the regularized Basis Pursuit problem (\ref{eq1.2}) with a linear rate
	$$
	\mathbb{E} D_f^{x_{k+1}^*}(x_{k+1},\hat{x})
	\leq 
	\left(
	1-\frac{1}{2}\cdot\frac{\tilde{\sigma}^2_{\min }(A)}{m}\cdot\frac{|\hat{x}|_{\min }}{|\hat{x}|_{\min }+2\lambda}
	\right)
	\mathbb{E}D_f^{x_k^*}(x_k,\hat{x}).
	$$
\end{theorem}
We now introduce Lemma 4.5 in \cite{lorenz2023minimal}, which will be useful in the following convergence analysis. 
\begin{lemma}[Lemma 4.5 in \cite{lorenz2023minimal}]
	\label{lemma1}
	Let $x,y,z\in\mathbb{R}^n$ such that $(y,z)$ is linearly independent. Then it holds that
	$$
	\|(I-P_{\langle \{y,z\} \rangle})x\|^2
	=\|(I-P_z)\circ(I-P_y)x\|^2-\frac{\langle (I-P_y)x,P_zy\rangle^2}{\|(I-P_{z})y\|^2}.
	$$
\end{lemma}
With the help of Theorem \ref{th4:2} and Lemma \ref{lemma1}, we are ready to present the convergence result of Algorithm \ref{al4}, which is similar to that of Theorem 4.6 in \cite{lorenz2023minimal}.
\begin{theorem}
	\label{th4.3}
Suppose that instead of exact data $b\in\mathcal{R}(A)$ only a corrupted right-hand side $\tilde{b}=b+b^c$ with $\|b^c\|_0=\beta m$ and $\|b^c\|_2\leq\delta$ is given in (\ref{eq1.2}). Assume that $0<\beta<q\leq 1-\beta$ and the matrix $A\in\mathbb{R}^{m\times n}$ has unit-norm rows. If the iterates $y_k,x_k,x_k^*=A^Ty_k\in\partial f(x_k)\cap \mathcal{R}(A)$ of Algorithm \ref{al4} are computed with $\tilde{b}$ and
	 $$\tilde{y}_k=\nabla \tilde{\Psi}(y_k)-e_{i_k}\nabla_{i_k} \tilde{\Psi}(y_k),$$ 
	then for any $\gamma>0$,
	it holds that
\begin{equation}
	\begin{aligned}
	\label{eq4:1}
	\mathbb{E}D_f^{x_{k+1}^*}(x_{k+1},\hat{x})
	&	\leq 
	\left(
	1-\frac{1}{2(2\gamma+\|A\|_2^2)}\cdot\frac{\tilde{\sigma}_{q-\beta,\min}^2}{qm}\cdot\frac{|\hat{x}|_{\min }}{|\hat{x}|_{\min }+2\lambda}
	\right)\mathbb{E}D_f^{x_k^*}(x_k,\hat{x})+\frac{1}{4\gamma}\delta^2
	\\
	&	-\frac{1}{2(2\gamma+\|A\|_2^2)}\mathbb{E}\left(\frac{\langle\tilde{y}_k,v_k^*\rangle^2}{\|v_k^*\|^2}
	+\frac{\langle v_k^*,e_{i_k}\nabla_{i_k} \tilde{\Psi}(y_k)\rangle^2
		\cdot
		\langle \tilde{y}_k,v_k^*\rangle^2}{\|v_k^*\|^2\cdot(\|v_k^*\|^2
		\|e_{i_k}\nabla_{i_k} \tilde{\Psi}(y_k)\|^2
		-\langle v_k^*,e_{i_k}\nabla_{i_k} \tilde{\Psi}(y_k)\rangle^2)}\right).
	\end{aligned}
\end{equation}	
\end{theorem}

\begin{proof}
	It follows from (\ref{eq5:6}) and the minimizing property of $\alpha_k$ and $w_k$ that we have
	\begin{equation}
	\label{eq4.8}
	\begin{aligned}
	&~~~~\Psi(y_{k+1})\\
	&\leq
	\Psi(y_k)+\langle -\alpha_k e_{i_k}\nabla_{i_k}\tilde{\Psi}(y_k)+w_kv_k^*, \nabla \tilde{\Psi}(y_k)\rangle +(\frac{\|A\|_2^2}{2}+\gamma)\|-\alpha_k e_{i_k}\nabla_{i_k}\tilde{\Psi}(y_k)+w_kv_k^*\|^2+\frac{1}{4\gamma}\delta^2\\
	&=\Psi(y_k)
	+ \frac{1}{2(2\gamma+\|A\|_2^2)}\left(\|\nabla \tilde{\Psi}(y_k) +(2\gamma+\|A\|_2^2)(-\alpha_k e_{i_k}\nabla_{i_k} \tilde{\Psi}(y_k)+w_kv_k^*)\|^2-\|\nabla \tilde{\Psi}(y_k)\|^2\right)+\frac{1}{4\gamma}\delta^2\\
	&\leq \Psi(y_k)+\frac{1}{2(2\gamma+\|A\|_2^2)}\inf_{\alpha,w}\left(\|\nabla \tilde{\Psi}(y_k) -\alpha e_{i_k}\nabla_{i_k} \tilde{\Psi}(y_k)+wv_k^*\|^2-\|\nabla \tilde{\Psi}(y_k)\|^2\right)+\frac{1}{4\gamma}\delta^2\\
	 &=\Psi(y_k)+\frac{1}{2(2\gamma+\|A\|_2^2)}\left(\|(I-P_{\langle \{e_{i_k}\nabla_{i_k} \tilde{\Psi}(y_k),v_k^*\}\rangle})
\nabla\tilde{\Psi}(y_k)\|^2-\|\nabla\tilde{\Psi}(y_k)\|^2\right)+\frac{1}{4\gamma}\delta^2.
\end{aligned}
\end{equation}
	We first calculate the expression 
	$$
	\begin{aligned}
	(I-P_{ e_{i_k}\nabla_{i_k} \tilde{\Psi}(y_k)})\nabla \tilde{\Psi}(y_k)
	&=\nabla \tilde{\Psi}(y_k)-\frac{\langle \nabla \tilde{\Psi}(y_k), e_{i_k}\nabla_{i_k} \tilde{\Psi}(y_k)\rangle}{\| e_{i_k}\nabla_{i_k} \tilde{\Psi}(y_k)\|^2}e_{i_k}\nabla_{i_k} \tilde{\Psi}(y_k)\\
	&=\nabla \tilde{\Psi}(y_k)-e_{i_k}\nabla_{i_k} \tilde{\Psi}(y_k).
	\end{aligned}
	$$
	Let $\tilde{y}_k=(I-P_{ e_{i_k}\nabla_{i_k} \tilde{\Psi}(y_k)})\nabla \tilde{\Psi}(y_k)$ and then we obtain
	\begin{equation}
	\label{eq4.5}
	\begin{aligned}
	\|(I-P_{v_k^*})\circ(I-P_{ e_{i_k}\nabla_{i_k} \tilde{\Psi}(y_k)})\nabla \tilde{\Psi}(y_k)\|^2
	&=\|\tilde{y}_k\|^2-\frac{\langle\tilde{y}_k,v_k^*\rangle^2}{\|v_k^*\|^2},    
	\end{aligned}
	\end{equation}
	and
	\begin{equation}
	\label{eq4.6}
	\begin{aligned}
&~~~~\frac{\langle (I-P_{ e_{i_k}\nabla_{i_k} \tilde{\Psi}(y_k)})\nabla \tilde{\Psi}(y_k), P_{v_k^*}e_{i_k}\nabla_{i_k} \tilde{\Psi}(y_k)\rangle^2}{\|(I-P_{v_k^*})e_{i_k}\nabla_{i_k} \tilde{\Psi}(y_k)\|^2}\\
	&=\frac{\langle \tilde{y}_k,\frac{\langle v_k^*,e_{i_k}\nabla_{i_k} \tilde{\Psi}(y_k)\rangle}{\|v_k^*\|^2}v_k^*\rangle^2 
	}{\|
		e_{i_k}\nabla_{i_k} \tilde{\Psi}(y_k)-\frac{\langle v_k^*,e_{i_k}\nabla_{i_k} \tilde{\Psi}(y_k)\rangle}{\|v_k^*\|^2}v_k^*
		\|^2}\\
	&=\frac{\langle v_k^*,e_{i_k}\nabla_{i_k} \tilde{\Psi}(y_k)\rangle^2
		\cdot
		\langle \tilde{y}_k,v_k^*\rangle^2}{\|v_k^*\|^2\cdot(\|v_k^*\|^2
		\|e_{i_k}\nabla_{i_k} \tilde{\Psi}(y_k)\|^2
		-\langle v_k^*,e_{i_k}\nabla_{i_k} \tilde{\Psi}(y_k)\rangle^2)}.
	\end{aligned}
	\end{equation}
 According to Lemma \ref{lemma1}, combining with (\ref{eq4.5}) and (\ref{eq4.6}), we have
	\begin{equation}
	\label{eq4.7}
	\begin{aligned}
	&~~~~\|(I-P_{\langle \{e_{i_k}\nabla_{i_k} \tilde{\Psi}(y_k),v_k^*\}\rangle})
	\nabla \tilde{\Psi}(y_k)\|^2\\
	&=\|(I-P_{v_k^*})\circ(I-P_{ e_{i_k}\nabla_{i_k} \tilde{\Psi}(y_k)})\nabla \tilde{\Psi}(y_k)\|^2-\frac{\langle (I-P_{ e_{i_k}\nabla_{i_k} \tilde{\Psi}(y_k)})\nabla \tilde{\Psi}(y_k), P_{v_k^*}e_{i_k}\nabla_{i_k} \tilde{\Psi}(y_k)\rangle^2}{\|(I-P_{v_k^*})e_{i_k}\nabla_{i_k} \tilde{\Psi}(y_k)\|^2}\\
	&=
	\|\tilde{y}_k\|^2-\frac{\langle\tilde{y}_k,v_k^*\rangle^2}{\|v_k^*\|^2}-
	\frac{\langle v_k^*,e_{i_k}\nabla_{i_k} \tilde{\Psi}(y_k)\rangle^2
		\cdot
		\langle \tilde{y}_k,v_k^*\rangle^2}{\|v_k^*\|^2\cdot(\|v_k^*\|^2
		\|e_{i_k}\nabla_{i_k} \tilde{\Psi}(y_k)\|^2
		-\langle v_k^*,e_{i_k}\nabla_{i_k} \tilde{\Psi}(y_k)\rangle^2)}.
	\end{aligned}
	\end{equation}
	Inserting (\ref{eq4.7}) into (\ref{eq4.8}), we deduce
	\begin{equation*}
	\begin{aligned}
	&\Psi(y_{k+1})\leq \Psi(y_k)-\frac{1}{2(2\gamma+\|A\|_2^2)}\|\nabla_{i_k} \tilde{\Psi}(y_k)\|^2+\frac{1}{4\gamma}\delta^2\\
	&~~~~~~~-
\frac{1}{2(2\gamma+\|A\|_2^2)}\left(\frac{\langle\tilde{y}_k,v_k^*\rangle^2}{\|v_k^*\|^2}
	+
	\frac{\langle v_k^*,e_{i_k}\nabla_{i_k} \tilde{\Psi}(y_k)\rangle^2
		\cdot
		\langle \tilde{y}_k,v_k^*\rangle^2}{\|v_k^*\|^2\cdot(\|v_k^*\|^2
		\|e_{i_k}\nabla_{i_k} \tilde{\Psi}(y_k)\|^2
		-\langle v_k^*,e_{i_k}\nabla_{i_k} \tilde{\Psi}(y_k)\rangle^2)}\right).
	\end{aligned}
	\end{equation*}
It follows from Lemma \ref{lemma3:1} and $\nabla_{i_k} \tilde{\Psi}(y_k)=a_{i_k}^Tx_k-\tilde{b}_{i_k}$ that we have
\begin{equation}
\label{eq4:7}
\begin{aligned}
&D_f^{x_{k+1}^*}(x_{k+1},\hat{x})\leq D_f^{x_{k}^*}(x_{k},\hat{x})-\frac{1}{2(2\gamma+\|A\|_2^2)}(a_{i_k}^Tx_k-\tilde{b}_{i_k})^2+\frac{1}{4\gamma}\delta^2\\
&~~~~~~~-
\frac{1}{2(2\gamma+\|A\|_2^2)}\left(\frac{\langle\tilde{y}_k,v_k^*\rangle^2}{\|v_k^*\|^2}
+
\frac{\langle v_k^*,e_{i_k}\nabla_{i_k} \tilde{\Psi}(y_k)\rangle^2
	\cdot
	\langle \tilde{y}_k,v_k^*\rangle^2}{\|v_k^*\|^2\cdot(\|v_k^*\|^2
	\|e_{i_k}\nabla_{i_k} \tilde{\Psi}(y_k)\|^2
	-\langle v_k^*,e_{i_k}\nabla_{i_k} \tilde{\Psi}(y_k)\rangle^2)}\right).
\end{aligned}
\end{equation}

Denote the set
consisting of all acceptable indices in the $k$-th iterate as $B$; the subset consisting of corrupted indices in $B$ is defined by $S$. It holds that $|B|=qm,~0\leq |S|\leq\beta m,~(q-\beta)m\leq|B\backslash S|\leq qm$.
Fix the indices $i_0,\cdots,i_{k-1}$ and only consider $i_k$ as a random variable, it follows from the conditional expectation that we have
\begin{equation}
\label{eq2:2}
\mathbb{E}_k(a_{i_k}^Tx_k-\tilde{b}_{i_k})^2
=P(i_k\in B\backslash S)\mathbb{E}_k[(a_{i_k}^Tx_k-\tilde{b}_{i_k})^2\mid i_k\in B\backslash S]
+
P(i_k\in  S)\mathbb{E}_k[(a_{i_k}^Tx_k-\tilde{b}_{i_k})^2\mid i_k\in S].
\end{equation}
For the first term of (\ref{eq2:2}), according to Lemma \ref{lemma3:2}, it holds that
$$
\begin{aligned}
\mathbb{E}_k[(a_{i_k}^Tx_k-\tilde{b}_{i_k})^2\mid i_k\in B\backslash S]
&=\frac{1}{|B\backslash S|}\|A_{B\backslash S}x_k-b_{B\backslash S}\|_2^2\\
&\geq \frac{1}{|B\backslash S|}\cdot\tilde{\sigma}_{q-\beta,\min}^2\cdot\frac{|\hat{x}|_{\min}}{|\hat{x}|_{\min}+2\lambda}D_f^{x_{k}^*}(x_{k},\hat{x}),
\end{aligned}
$$
where $\tilde{\sigma}_{q-\beta,\min}$ is defined in (\ref{2:1}).
Consequently, (\ref{eq2:2}) can be reformulated into
$$
\mathbb{E}_k(a_{i_k}^Tx_k-\tilde{b}_{i_k})^2\geq 
\frac{\tilde{\sigma}_{q-\beta,\min}^2}{qm}\cdot\frac{|\hat{x}|_{\min }}{|\hat{x}|_{\min }+2\lambda}D_f^{x_{k}^*}(x_{k},\hat{x}).
$$	
Together with Lemma \ref{lemma3:1},
considering all indices $i_0,\cdots,i_k$ as random variables and taking full expectation of (\ref{eq4:7}) gives
$$
	\begin{aligned}
\mathbb{E}D_f^{x_{k+1}^*}(x_{k+1},\hat{x})
	&	\leq 
	\left(
	1-\frac{1}{2(2\gamma+\|A\|_2^2)}\cdot\frac{\tilde{\sigma}_{q-\beta,\min}^2}{qm}\cdot\frac{|\hat{x}|_{\min }}{|\hat{x}|_{\min }+2\lambda}
	\right)\mathbb{E}D_f^{x_k^*}(x_k,\hat{x})+\frac{1}{4\gamma}\delta^2
	\\
&	-\frac{1}{2(2\gamma+\|A\|_2^2)}\mathbb{E}\left(\frac{\langle\tilde{y}_k,v_k^*\rangle^2}{\|v_k^*\|^2}
+\frac{\langle v_k^*,e_{i_k}\nabla_{i_k} \tilde{\Psi}(y_k)\rangle^2
		\cdot
		\langle \tilde{y}_k,v_k^*\rangle^2}{\|v_k^*\|^2\cdot(\|v_k^*\|^2
		\|e_{i_k}\nabla_{i_k} \tilde{\Psi}(y_k)\|^2
		-\langle v_k^*,e_{i_k}\nabla_{i_k} \tilde{\Psi}(y_k)\rangle^2)}\right).
	\end{aligned}
	$$
The proof is completed.
\end{proof}

\begin{remark}
\label{remark3.3}
\begin{enumerate}[(a)]
	\item According to Cauchy-Schwartz's inequality, we have 
	$$
	\|v_k^*\|^2
	\|e_{i_k}\nabla_{i_k} \tilde{\Psi}(y_k)\|^2
	-\langle v_k^*,e_{i_k}\nabla_{i_k} \tilde{\Psi}(y_k)\rangle^2\geq 0,
	$$
	thus the term in second line of (\ref{eq4:1}) in Theorem \ref{th4.3} is negative.
	\item Theorem \ref{th4.3} establishes an upper bound for Quantile-RaSK-MM with respect to Bregman distance, which is composed of the linear convergence term and convergence horizon term. 
	There holds $D_{f}^{x_k^*}(x_k,\hat{x})\rightarrow 0$ as $\delta\rightarrow 0$ and $k\rightarrow \infty$.	
The convergence horizon $\frac{1}{4\gamma}\delta^2$ caused by corruptions will be considered and addressed in our future work, which should be a valuable topic to break through the convergence horizon. 
	\item The selection of the positive parameter $\gamma$ involves a trade-off between the convergence rate and the convergence horizon. 
	We explore the effect of the value of $\gamma$ on Quantile-RaSK-MM by extensive numerical tests.
\end{enumerate}

\end{remark}

\subsection{The convergence analysis of RaSK-MM}
\label{sec4.2}

We now present convergence results of RaSK-MM in noisy and exact cases, respectively, which is a special instance of Quantile-RaSK-MM in the case of $\beta=0$ and $q=1$.

\subsubsection{Extensions to noisy linear systems}

In practical applications, the measurement right-hand vector in the linear system $Ax=b$ is often contaminated by small but widespread noise.
Suppose that the noisy vector $\tilde{b}$ is of the form $$\tilde{b}=b+r,$$ 
where $b=A\hat{x}$ is the exact data, $r$ is the additive noise and $\|r\|\leq \delta$.
The convergence result of RaSK-MM is shown in Theorem \ref{th4.5}, which can be directly derived from Theorem \ref{th4.3}.

\begin{theorem}
	\label{th4.5}
	Assume that instead of the exact right-hand side $b$ only a noisy right-hand side $\tilde{b}=b+r$ is given with $\|r\|\leq \delta$ in (\ref{eq1.2}).
	Assume that $0<\beta<q\leq 1-\beta$ and the matrix $A\in\mathbb{R}^{m\times n}$ has unit-norm rows. If the iterates $y_k,x_k,x_k^*=A^Ty_k\in\partial f(x_k)\cap \mathcal{R}(A)$ of RaSK-MM are computed with $\tilde{b}$ and
	$$\tilde{y}_k=\nabla \tilde{\Psi}(y_k)-e_{i_k}\nabla_{i_k} \tilde{\Psi}(y_k).$$ 
	Then for any $\gamma>0$,
	it holds that
	$$
	\begin{aligned}
	&\mathbb{E}D_f^{x_{k+1}^*}(x_{k+1},\hat{x})
	\leq 
	\left(
	1-\frac{1}{2(2\gamma+\|A\|_2^2)}\cdot\frac{\tilde{\sigma}^2_{\min }(A)}{m}\cdot\frac{|\hat{x}|_{\min }}{|\hat{x}|_{\min }+2\lambda}
	\right)
	\mathbb{E}D_f^{x_k^*}(x_k,\hat{x})\\
	&~~~~~~~	-\frac{1}{2(2\gamma+\|A\|_2^2)}	\mathbb{E}\left(\frac{\langle v_k^*,e_{i_k}\nabla_{i_k} \tilde{\Psi}(y_k)\rangle^2
		\cdot
		\langle \tilde{y}_k,v_k^*\rangle^2}{\|v_k^*\|^2\cdot(\|v_k^*\|^2
		\|e_{i_k}\nabla_{i_k} \tilde{\Psi}(y_k)\|^2
		-\langle v_k^*,e_{i_k}\nabla_{i_k} \tilde{\Psi}(y_k)\rangle^2)}
	+
\frac{\langle\tilde{y}_k,v_k^*\rangle^2}{\|v_k^*\|^2}
	\right)
	+\frac{1}{4\gamma}\delta^2.
	\end{aligned}
	$$
	
\end{theorem}

\begin{remark}
	\begin{enumerate}[(a)]
		\item In noisy case, we have $q=1$ and $\beta=0$, then the singular value $\tilde{\sigma}_{q-\beta,\min}^2(A)$ in Theorem \ref{th4.3} reduces to $\tilde{\sigma}_{\min}^2(A)$ in Theorem \ref{th4.5}. If the matrix $A$ has full column rank, we have $\tilde{\sigma}_{\min}^2(A)=\sigma_{\min}^2(A)$.
		\item When $2\gamma+\|A\|_2^2>1$, the convergence rate of RaSK-MM is slower than that of RaSK presented in \cite{schopfer2019linear} while the convergence horizon of RaSK-MM can be smaller than that of RaSK by selecting an appropriate parameter $\gamma$.
		\item When $\lambda=0$, 
		RaSK-MM reduces to the randomized Kaczmarz with minimal dual function (RK-MM) method; correspondingly, Quantile-RaSK-MM reduces to Quantile-RK-MM. The improvement of Quantile-RK-MM is the ability to go beyond the exact measurement restriction in \cite{zeng2024adaptiveh,lorenz2023minimal}. 
	\end{enumerate}
\end{remark}



\subsubsection{Extensions to exact linear systems}
\label{sec4.1}

The RaSK-MM with $\gamma=0$ can be employed to solve sparse solutions of the linear system $Ax=b$ with exact right-hand vector $b$.
We turn to study the convergence result of RaSK-MM for solving exact linear systems, which can be induced directly from Theorem \ref{th4.5}.


\begin{theorem}
	\label{th3:4}
	Assume that the exact right-hand side $b$ is given in (\ref{eq1.2}).
	Let $0<\beta<q\leq 1-\beta$ and the matrix $A\in\mathbb{R}^{m\times n}$ has unit-norm rows. 
	Let $y_k,x_k,x_k^*=A^Ty_k\in\partial f(x_k)\cap \mathcal{R}(A)$ be the iterates of RaSK-MM with $\gamma=0$, then $x_k$ converges linearly in expectation to the unique solution $\hat{x}$ of the regularized Basis Pursuit problem (\ref{eq1.2}). Let  
	$$\tilde{y}_k=\nabla \tilde{\Psi}(y_k)-e_{i_k}\nabla_{i_k} \tilde{\Psi}(y_k),$$ 
	then it holds that
	$$
	\begin{aligned}
	\mathbb{E}D_f^{x_{k+1}^*}(x_{k+1},\hat{x})
	&	\leq 
	\left(
	1-\frac{1}{\|A\|_2^2}\cdot\frac{\tilde{\sigma}^2_{\min }(A)}{2m}
	\cdot\frac{|\hat{x}|_{\min }}{|\hat{x}|_{\min }+2\lambda}
	\right)\mathbb{E}D_f^{x_k^*}(x_k,\hat{x})
\\
	&~~~~~~~~-\frac{1}{2\|A\|_2^2}\mathbb{E}\left(\frac{\langle v_k^*,e_{i_k}\nabla_{i_k} \Psi(y_k)\rangle^2
		\cdot
		\langle \tilde{y}_k,v_k^*\rangle^2}{\|v_k^*\|^2\cdot(\|v_k^*\|^2
		\|e_{i_k}\nabla_{i_k} \Psi(y_k)\|^2
		-\langle v_k^*,e_{i_k}\nabla_{i_k} \Psi(y_k)\rangle^2)}
	+\frac{\langle\tilde{y}_k,v_k^*\rangle^2}{\|v_k^*\|^2}
	\right).
	\end{aligned}
	$$
\end{theorem}


\begin{remark}
\label{remark4.3}
We compare RaSK-MM with existing methods in terms of computational complexity and convergence results. 
\begin{enumerate}[(a)]
\item Under the assumption that the measurement matrix is normalized by rows, the computational cost of variants of the sparse Kaczmarz method is summarized in Table \ref{table7}. 
RaSK-MM significantly more expensive than other methods in the overdeterminated case ($m>n$). Compared with BK-EM and BK-REM, RaSK-MM is able to solve noisy linear systems.
\begin{table}[H]
	\begin{center}
		\begin{minipage}{\textwidth}
			\caption{Summary of the computational costs for variants of sparse Kaczmarz method, where $m,n$ are the numbers of rows and columns of the measurement matrix, BK-EM and BK-REM are equipped with single row sketching}
			\label{table7}
			\begin{tabular*}{\textwidth}{@{\extracolsep{\fill}}lcc@{\extracolsep{\fill}}}
				\toprule
		Method  &  Computational cost & Rate bound are shown in
				\\	
				\midrule RaSK  & $9n$    &  Theorem 3.2 in \cite{schopfer2019linear}
				\\
			ERaSK & $19n+\mathcal{O}(nlog n)$  & Theorem 3.2 in \cite{schopfer2019linear} 
				\\
			BK-EM  & $24n+\mathcal{O}(nlog n)$  & Theorem 4.2 in \cite{lorenz2023minimal}
				\\
			BK-REM  & $17n$  &Theorem 4.5 in \cite{lorenz2023minimal} 
				\\
	RaSK-MM  & $4mn+7m+4n$  &  Section \ref{sec4.2} in this paper
				\\
				\bottomrule
			\end{tabular*}
		\end{minipage}
	\end{center}
\end{table}
	\item When $\|A\|_2^2>1$, the convergence rate of RaSK-MM with $\gamma=0$ is slower than that of RaSK and BK-REM. However, the convergence result of RaSK-MM has a decreasing term compared to RaSK, which may lead to better performance. Moreover, it follows from the inequality $\|A^Tv_k^*\|\leq \|A\|\|v_k^*\|$ that RaSK-MM has a smaller descent compared to BK-REM in \cite{lorenz2023minimal}. 
\end{enumerate} 
\end{remark}

\section{Numerical experiments}
\label{sec5}
\subsection{Experimental setup}
\label{sec5:1}

In this section, we evaluate Quantile-RaSK-MM and RaSK-MM on three types of coefficient matrices. 
The first type is the random Gaussian matrix $A\in\mathbb{R}^{m\times n}$ generated by using MATLAB function ‘randn’, where the elements are independent and identically distributed from the standard normal distribution $N(0,1)$. The second type is selected from the SuiteSparse Matrix Collection in \cite{davis2011university}. 
In these two cases, the sparse solution $\hat{x}\in\mathbb{R}^n$ is randomly generated by using MATLAB function 'randn' as well, with the nonzero locations chosen randomly by the sparsity $s$; the exact measurement vector $b\in\mathbb{R}^m$ is computed by $b=A\hat{x}$. 
The third type uses tomography matrices produced by the function \textsl{paralleltomo} in AIRTOOLS package in \cite{hansen2012air,hansen2018air}, where the measurement matrix $A$, the true solution $\hat{x}$ (ground image), and the exact measurement vector $b$ are generated with parameters $N\in\{20,30\},~theta=0:5:179,~p=150,~d=1.4N$.

The corrupted right-hand vector is the sum of $b$ and corruptions $b^c$ from the uniform distribution $U(-100,100)$, while the noisy right-hand vector is the sum of $b$ and noise $r$ from $N(0,1)$.
All implementations start from the initial vector $x_{0}=0\in \mathbb{R}^{n}$ and $y_{-1}=y_0=0\in\mathbb{R}^{m}$.
We define the relative error at $k$-th iterate as
$$\text{Error}(k)=\frac{\|x_k-\hat{x}\|}{\|\hat{x}\|}.$$
Unless specified otherwise, we record the median of residual error through 50 trials by default; the algorithms terminate once the maximum number of iterations $N$ is reached or the relative error is less than the accuracy $\varepsilon$ we set.

All experiments are performed with MATLAB (version R2021b) on a personal
computer with 2.80GHZ CPU(Intel(R) Core(TM) i7-1165G7), 16-GB memory, and
Windows operating system (Windows 10).

\subsection{Parameter tuning for $\beta$ and $q$}
\label{sec5.2}
 
Based on the inequality $\beta<q\leq1-\beta$, we begin by exploring the optimal choice of the quantile $q$ for a range of corrupted fractions $\beta$ for Quantile-RaSK-MM.  
The experiments are conducted on a random matrix generated by MATLAB function 'randn', and the measurement vector $\tilde{b}$ is contaminated by corruptions from the uniform distribution $U(-100,100)$.
In Figure \ref{fig1}, we set $\gamma = 0.01,\lambda = 1,s=10$ and $\beta=0.1:0.1:0.5,q=0.1:0.1:1$. For each $\beta$, we plot the relative error after 2000 iterations of Quantile-RaSK-MM, varying different values of $q$. 

Figure \ref{fig1} (a) presents the results for underdetermined matrices $A\in\mathbb{R}^{500\times 2000}$, while Figure \ref{fig1} (b) show the results for overdetermined matrices $A\in\mathbb{R}^{500\times 200}$. The findings in Figure \ref{fig1} indicate that, in both underdetermined and overdetermined cases, the optimal quantile $q$ can be taken close to $1-\beta$. This is consistent with observations reported in \cite{zhang2024quantile,cheng2022block,haddock2022quantile}. As discussed in Remark \ref{remark1}, the acceptable ratio $q$ can be taken as $1-\beta$ in our proposed Quantile-RaSK-MM, whereas the quantile $q$ must be strictly less than $1-\beta$ in the methods proposed by \cite{zhang2024quantile,cheng2022block,haddock2022quantile}. This flexibility in the choice of $q$ represents a significant advantage of our method.

\begin{figure}[H]
	\centering
	\subfigure[underdetermined system]{
		\includegraphics[width=0.3\linewidth,height=0.23\textwidth]{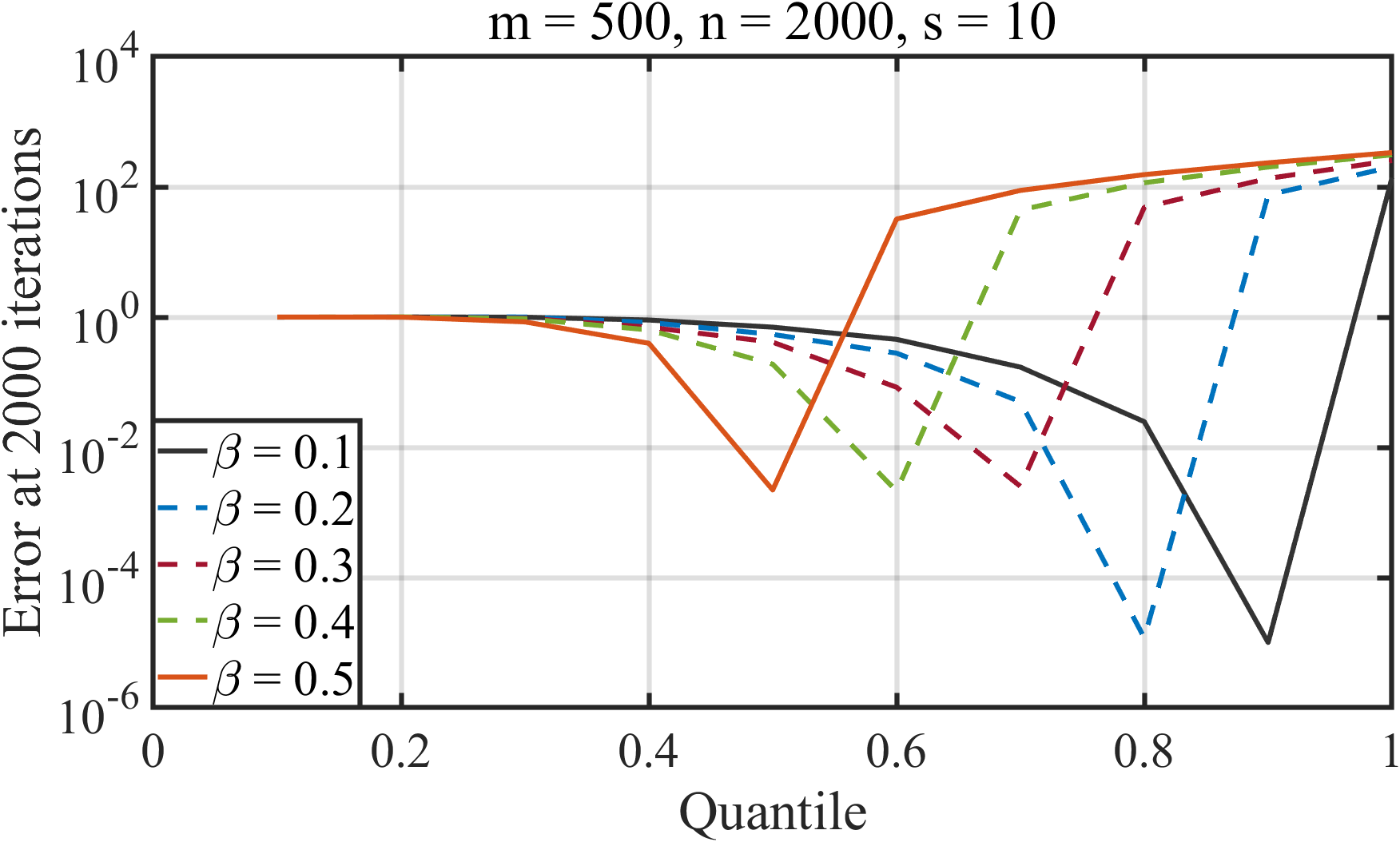}}
	\hspace{0.06\linewidth}
	\subfigure[overdetermined system]{\includegraphics[width=0.3\linewidth,height=0.23\textwidth]{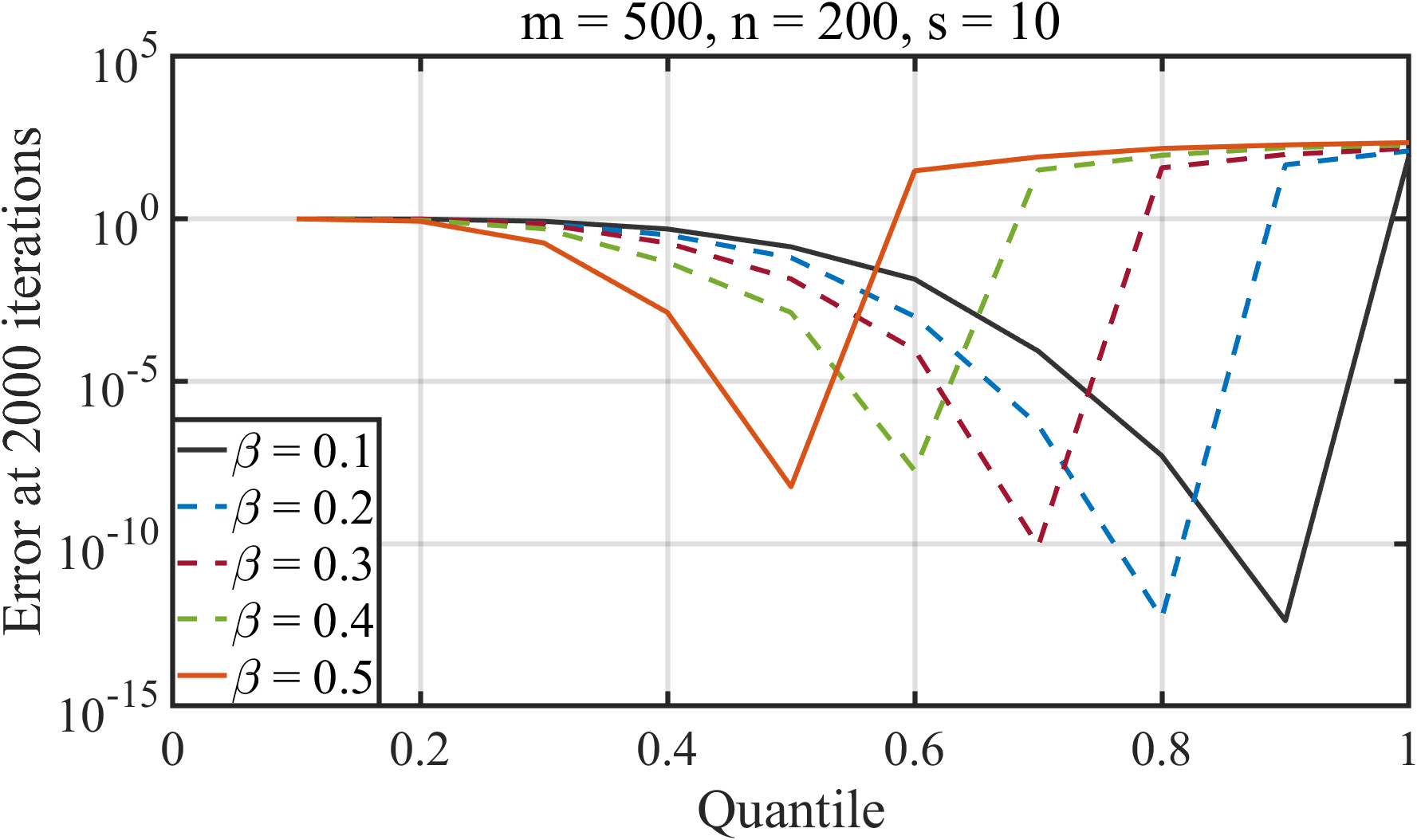}}
	\caption{Relative error after 2000 iterations of Quantile-RaSK-MM for a range of $\beta$ and $q$}
	\label{fig1}
\end{figure}

\subsection{The simulated data}

\label{sec5.1}
In this subsection, we investigate the performance of RaSK-MM and Quantile-RaSK-MM on linear systems with randomly generated Gaussian matrices.

\subsubsection{The noisy linear systems}

In this simulation, the noisy linear systems are constructed with Gaussian matrices $A\in\mathbb{R}^{m\times n}$ and the vector $b$ is damaged by $0.1\%$ relative noise.
We compare the number of iterations (IT) and computing time (CPU) of RaSK, ERaSK, and RaSK-MM with $\gamma$ chosen from the set $\{0.01,0.1\}$.
The iteration is terminated once the relative error is less than $\varepsilon=10^{-2}$ or the number of iterations exceeds $N=20000$. 
Table \ref{table6} shows the median of different methods through 50 trials, both with $s=10$ and $\lambda=1$. 
In the underdetermined case, RaSK-MM achieves given accuracy in the least time, followed by RaSK; in the overdetermined case, RaSK is fastest to reach the given accuracy. The more highly underdetermined the linear system, the better the RaSK-MM performs, which is consistent with Remark \ref{remark4.3}.

\begin{table}[H]
	\begin{center}
		\begin{minipage}{\textwidth}
			\caption{IT and CPU of different methods for random matrices when Error$<10^{-2}$ in noisy case}
			\label{table6}
			\begin{tabular*}{\textwidth}{@{\extracolsep{\fill}}lcccccc@{\extracolsep{\fill}}}
				\toprule
				\multicolumn{2}{c}{ $m\times n$ }  
				&  $200\times 1000$ & $200\times 500$ & $200\times 200$ &$500\times 200$ &$1000\times 200$
				\\	
				\midrule
				\multirow{2}{*}{ RaSK \cite{schopfer2019linear}} & IT     &   16566   & 8090 & 2452 &1857 & 2027
				\\
				\cline{2-7}
				& CPU    & 0.0559 & 0.0147  &  \textbf{0.0033} &\textbf{0.0062} &\textbf{0.0075}
				\\
				\midrule
				\multirow{2}{*}{ ERaSK \cite{schopfer2019linear} } & IT  &  1317  &  908   &  544 & 485 & 437
				\\
				\cline{2-7}
				& CPU &   0.0846  & 0.0323 &  0.0112 & 0.0272 & 0.0274
				\\
				\midrule
				\multirow{2}{*}{ RaSK-MM with $\gamma=0.01$ } & IT   &  1406  & 868 &  603 &552 & 595
				\\
				\cline{2-7}
				& CPU   &  \textbf{0.0323}  & \textbf{0.0105}  &  0.0047 &0.0175 & 0.0413
				\\
				\midrule
				\multirow{2}{*}{ RaSK-MM with $\gamma=0.1$ } & IT   &  1376  & 895 &  586 &554 & 641
				\\
				\cline{2-7}
				& CPU & 0.0356  & 0.0126 & 0.0052 &0.0183 & 0.0479
				\\
				\bottomrule
			\end{tabular*}
		\end{minipage}
	\end{center}
\end{table}


\subsubsection{The corrupted linear systems}


We evaluate the performance of Quantile-RaSK-MM on corrupted linear systems with Gaussian matrices. 
Let $\lambda=1,s=10$. In Quantile-RaSK and Quantile-ERaSK, as recommended in \cite{zhang2024quantile,cheng2022block,haddock2022quantile}, we let $q=0.7,\beta=0.2$; in Quantile-RaSK-MM, we let $q=0.8,\beta=0.2$ as shown in Section \ref{sec5.2}.
If the algorithm fails to converge to the accuracy $10^{-6}$ within $20000$ iterations, it is denoted as '-'. 

Table \ref{table3} and \ref{table4} list the median of IT and CPU for underdetermined ($m<n$) and overdetermined ($m>n$) linear systems through 50 trials, respectively.
We observe that, for underdetermined systems, the Quantile-RaSK-MM method with $\gamma=0.01$ attains the best performance in terms of both number of iterations and computing time, whereas Quantile–RaSK and Quantile–ERaSK fail to converge in highly underdetermined cases.
For highly overdetermined systems, Quantile-ERaSK reaches the given accuracy faster than Quantile-RaSK-MM.

In summary, the Quantile-RaSK-MM method demonstrates robust and efficient performance in solving corrupted linear systems, particularly in underdetermined scenarios.


 \begin{table}[H]
 	\begin{center}
 		\begin{minipage}{\textwidth}
 			\caption{IT and CPU of different methods for random underdetermined matrices when Error < $10^{-6}$ in corrupted case}
 			\label{table3}
 			\begin{tabular*}{\textwidth}{@{\extracolsep{\fill}}lccccc@{\extracolsep{\fill}}}
 				\toprule
 				\multicolumn{2}{c}{ $m\times n$ }  
 				& $500\times 1000$   & $500\times 2000$ & $500\times 3000$ & $500\times 4000$	
 				\\	
 				\midrule
 				\multirow{2}{*}{ Quantile-RaSK \cite{zhang2024quantile} } & IT &   -   &   -  &  - &-
 				\\
 				\cline{2-6}
 				& CPU &   -     &  -   & - &-
 				\\
 				\midrule
 				\multirow{2}{*}{ Quantile-ERaSK \cite{zhang2024quantile} } & IT &  12252   &  -  &  -  &-
 				\\
 				\cline{2-6}
 				& CPU &   2.0730  &  -   &  - &-
 				\\
 				\midrule
 				\multirow{2}{*}{ Quantile-RaSK-MM, $\gamma=0.01$ } & IT &  1968    &  2263   & 2672  & 3318
 				\\
 				\cline{2-6}
 				& CPU &  \textbf{0.2414}  &  \textbf{0.4948} & \textbf{0.9968}  & \textbf{2.4772}
 				\\
 				\midrule
 				\multirow{2}{*}{Quantile-RaSK-MM, $\gamma=0.1$}  & IT &  1979    &  2390  &  2769 & 3408
 				\\
 				\cline{2-6}
 				& CPU &  0.2428 &0.5250  &  1.0326 & 2.5422
 				\\
 				\bottomrule
 			\end{tabular*}
 		\end{minipage}
 	\end{center}
 \end{table}
\begin{table}[H]
	\begin{center}
		\begin{minipage}{\textwidth}
			\caption{IT and CPU of different methods for random overdetermined matrices when Error < $10^{-6}$ in corrupted case}
			\label{table4}
			\begin{tabular*}{\textwidth}{@{\extracolsep{\fill}}lccccc@{\extracolsep{\fill}}}
				\toprule
				\multicolumn{2}{c}{ $m\times n$ }  
				& $500\times 200$   &$1000\times 200$ &$2000\times 200$ &$4000\times 200$	
				\\	
				\midrule
				\multirow{2}{*}{ Quantile-RaSK \cite{zhang2024quantile} } & IT & 11616      &9543   & 7932 & 7880
				\\
				\cline{2-6}
				& CPU &  0.4860    &1.7354 & 2.6452 &5.9208
				\\
				\midrule
				\multirow{2}{*}{ Quantile-ERaSK \cite{zhang2024quantile} } & IT &  1666   &  1286  &  1210  & 1218
				\\
				\cline{2-6}
				& CPU &   0.1164   &   0.3548  & \textbf{0.5567} & \textbf{1.0602}
				\\
				\midrule
				\multirow{2}{*}{ Quantile-RaSK-MM, $\gamma=0.01$ } & IT &  1285    &  1363   &  1656 & 2031
				\\
				\cline{2-6}
				& CPU &   \textbf{0.0635}  &   \textbf{ 0.3055}  & 0.7966 & 2.0270
				\\
				\midrule
				\multirow{2}{*}{Quantile-RaSK-MM, $\gamma=0.1$}  & IT &    1319   & 1382  &  1569 &  2139
				\\
				\cline{2-6}
				& CPU &  0.0654  &0.3099 & 0.7536 &2.1382
				\\
				\bottomrule
			\end{tabular*}
		\end{minipage}
	\end{center}
\end{table}



\subsubsection{The stopping rule}
\label{sec5.3.3}

Here we compare the numerical performance of the two stopping criteria proposed in Section \ref{sec3.3}, namely ME and DP. 
We record $D_f(x_k,\hat{x})$ instead of $\|x_k-\hat{x}\|$ in each iterate to align with the motivation behind the ME rule, aiming to ensure monotone decrease of the error measured by Bregman distance.
The noisy linear systems are constructed with Gaussian matrices $A\in\mathbb{R}^{m\times n}$ and the right-hand side vector $b$ is contaminated by $5\%$ relative noise.
The corrupted linear systems are also constructed with Gaussian matrices $A\in\mathbb{R}^{m\times n}$ and $b$ is damaged by corruptions with a ratio $\beta=0.2$.
Let $q=0.8,s=10,\lambda=1,\gamma=0.01$.

Figure \ref{fig14} records the median of Bregman distance $\{D_f^{x_k^*}(x_k,\hat{x})\}$ over 50 independent trials; vertical markers indicate the average stopping iterations for DP and ME, respectively.
We find that both ME and DP succeed in terminating the RaSK-MM method effectively for noisy systems. However, only the ME rule yields a valid stopping behavior for the Quantile-RaSK-MM method in the corrupted case, since the residual sequence 
$\{\|Ax_k-\tilde{b}\|\}_k$  becomes too large to serve as a reliable stopping rule in corrupted scenarios.
In summary, as a stopping rule for Quantile-RaSK-MM, the ME rule is suitable for both noisy and corrupted linear systems, whereas DP is applicable only to noisy systems.

\begin{figure}[H]
	\centering
	\subfigure[Error in noisy case]{
		\includegraphics[width=0.23\linewidth,height=0.2\textwidth]{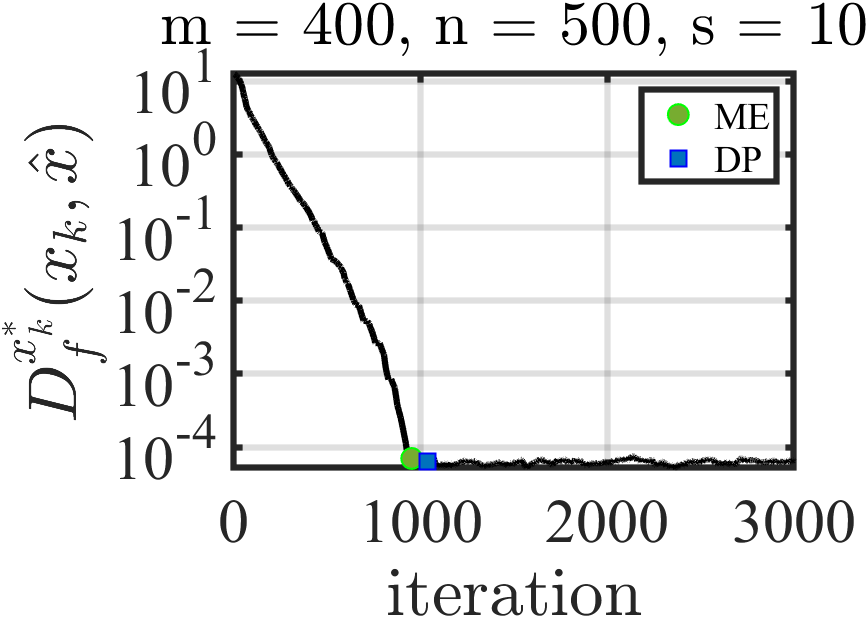}}
	\subfigure[Error in corrupted case]{\includegraphics[width=0.23\linewidth,height=0.2\textwidth]{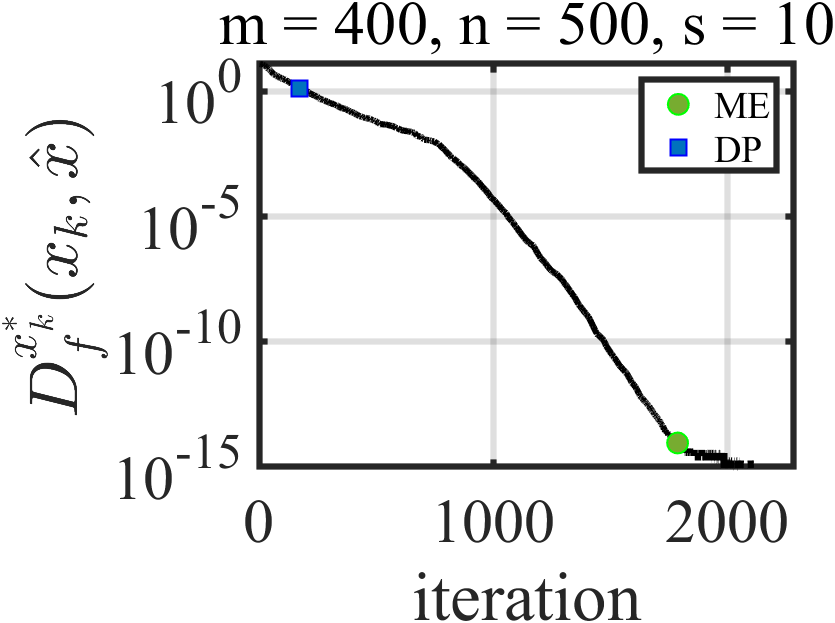}}
	\subfigure[Error in noisy case]{
		\includegraphics[width=0.23\linewidth,height=0.2\textwidth]{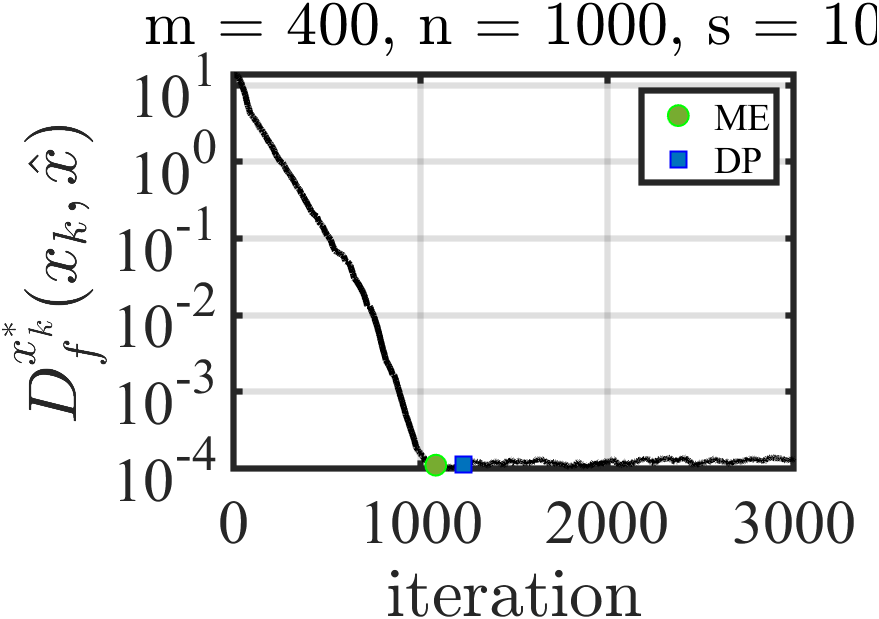}}
	\subfigure[Error in corrupted case]{\includegraphics[width=0.23\linewidth,height=0.2\textwidth]{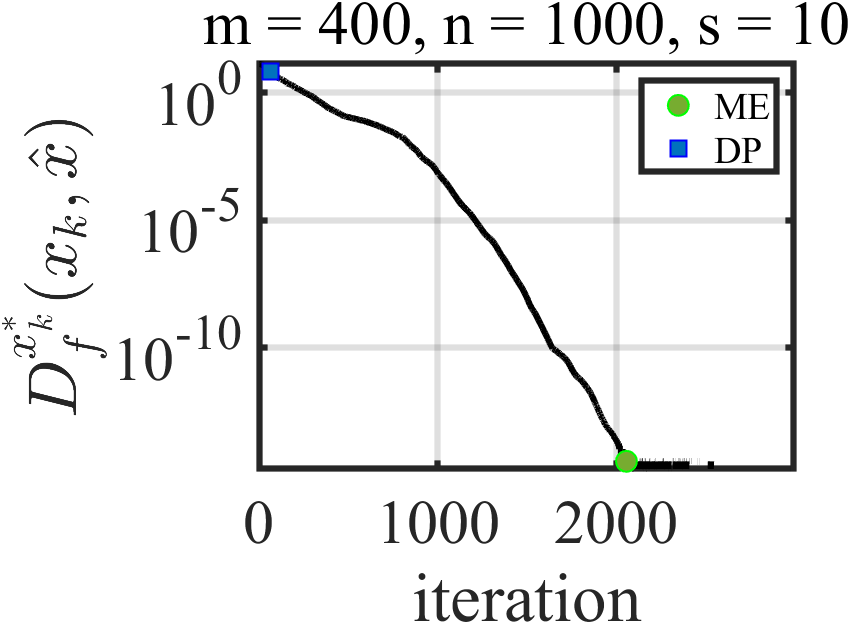}}
	\caption{The performance of the minimal error (ME) rule and the discrepancy principle (DP) on Quantile-RaSK-MM for solving noisy or corrupted systems}
	\label{fig14}
\end{figure}

\subsection{The SuiteSparse Collection data}
\label{sub5.2}

In this subsection, we test Quantile-RaSK-MM and RaSK-MM on real-world matrices selected from the SuiteSparse Matrix Collection in \cite{davis2011university}.

\subsubsection{The exact linear systems}

\label{sec5.2.1}

The matrices used in this experiment are 'ash219', 'Trefethen\_300', 'Trefethen\_700', and 'flower\_5\_1'. For each matrix, we construct exact linear systems and generate the $s$-sparsity truth solution $\hat{x}$ with $s=30,s=30,s=70,s=20$, respectively, and set $b=A\hat{x}$.
We evaluate BK-EM and BK-REM as described in \cite{lorenz2023minimal}, RaSK and ERaSK as detailed in \cite{schopfer2019linear}, RaSK-MM with $\gamma=0$ on these systems. The regularization parameter is set to $\lambda=1$.

Table \ref{table1} records the median over 50 trials, with termination at $\varepsilon=10^{-6}$ or $N=20000$.
As discussed in Remark \ref{remark4.3}, although BK-REM theoretically has a larger descent than RaSK-MM, our results show that RaSK-MM outperforms BK-REM on each real matrix used in this test. 
As an accelerated algorithm equipped with adaptive randomized heavy ball momentum, both BK-EM and RaSK-MM have the ability to improve the performance of RaSK.
Thus, even compared to BK-EM and BK-REM, RaSK-MM is a feasible and effective method to recover the sparse solution from exact linear systems.

\begin{table}[H]
	\begin{center}
		\begin{minipage}{\textwidth}
			\caption{IT and CPU of different methods for real matrices when Error < $10^{-6}$ in the exact case}
			\label{table1}
			\begin{tabular*}{\textwidth}{@{\extracolsep{\fill}}lccccc@{\extracolsep{\fill}}}
				\toprule
				\multicolumn{2}{c}{ Name } & ash219  &  Trefethen\_300 & Trefethen\_700& flower\_5\_1 \\
				\midrule
				\multicolumn{2}{c}{ $m\times n$ }  
				  & $219\times 85$   & $300\times 300$ & $700\times 700$& $211\times 201$	
				\\	
				\multicolumn{2}{c}{ density }  &  2.3529\%
			   & 5.1977\% & 2.5824\% & 1.42\%
				\\
				\multicolumn{2}{c}{cond($A$)} & 3.0249
			  & 1772.69 & 4710.40 & 2.000742e+16
				\\
				\multicolumn{2}{c}{rank($A$)} & 85     & 300 & 700 &179
				\\
				\midrule
		\multirow{2}{*}{ BK-EM\cite{lorenz2023minimal} } & IT & 1777    &  4626  & 12380  & 10173
		\\
		\cline{2-6}
		& CPU & 0.0132      &  0.1403  & 0.2567 &\textbf{0.0811}
		\\
		\midrule
		 \multirow{2}{*}{ BK-REM\cite{lorenz2023minimal} } & IT & 4224    &  11842  &   -
		 &-\\
		 \cline{2-6}
		 & CPU & 0.0258       &  0.2148
		  & - 
		  &-
		 \\
		 \midrule
		 \multirow{2}{*}{ RaSK\cite{schopfer2019linear} } & IT & 4455    &  13493  &   -&-
		 \\
		 \cline{2-6}
		 & CPU & 0.0193  &  0.1212
		   & -  
		   &-
		 \\
		 \midrule
		 \multirow{2}{*}{ ERaSK\cite{schopfer2019linear} } & IT & 2246    &  4024  &  10277 & 15491
		 \\
		 \cline{2-6}
		 & CPU & 0.0281       &  0.1617  & 0.4008 & 0.1767
		 \\
		 \midrule
		 \multirow{2}{*}{ RaSK-MM } & IT & 2792    &  4030 &   9826 & 17861
		 \\
		 \cline{2-6}
		 & CPU & \textbf{0.0108}      &  \textbf{0.0537}  & \textbf{0.1961} &0.0880
		 \\
				\bottomrule
			\end{tabular*}
		\end{minipage}
	\end{center}
\end{table}


\subsubsection{The corrupted linear systems}

Now we verify the performance of Quantile-RaSK-MM dealing with corrupted linear systems constructed by matrices from the SuiteSparse Matrix Collection. The sparsity is set to $s=30$ for 'ash219', 'ash608', and 'ash958'; the sparsity is set to $s=20$ for 'Franz1'. Let $\beta=0.2$ and $\lambda=1$. The candidates of $\gamma$ are chosen from $\{0.01,0.1\}$.
As discussed in Remark \ref{remark1}, the acceptable ratio $q$ in Quantile-RaSK and Quantile-ERaSK is set to $q=0.799$, and in Quantile-RaSK-MM it is set to $q=0.8$. 
We summarize the median of 50 trials of CPU and IT in Table \ref{table2} once the relative error reaches $\varepsilon=10^{-6}$ or the maximum number of iterations reaches $N=20000$. Selected numerical results are visualized in Figures \ref{fig9}.

We can find that Quantile-RaSK-MM has advantages over other algorithms, which requires the fewest iterations and computing time.
Although the stepsize parameter and momentum parameter in Quantile-RaSK-MM demand a larger computing cost at each iteration, their overall iteration can overcome this drawback and lead to a lower computing time than Quantile-RaSK and Quantile-ERaSK.

\begin{table}[H]
	\begin{center}
		\begin{minipage}{\textwidth}
			\caption{IT and CPU of different methods for real matrices when Error < $10^{-6}$ in corrupted case}
			\label{table2}
			\begin{tabular*}{\textwidth}{@{\extracolsep{\fill}}lccccc@{\extracolsep{\fill}}}
				\toprule
				\multicolumn{2}{c}{ Name } & ash219 & ash608  &  ash958 &Franz1 \\
				\midrule
				\multicolumn{2}{c}{ $m\times n$ }  & $219\times 85$ 
				& $608\times 188$   & $958\times 292$ 	& $2240\times 768$
				\\	
				\multicolumn{2}{c}{ density }  &  2.3529\%
				& 1.0638\%   & 0.6849\% & 0.3\% 
				\\
				\multicolumn{2}{c}{cond($A$)} & 3.0249
				&  3.373   & 3.2014 & 2.742609e+15
				\\
				\multicolumn{2}{c}{rank($A$)} & 85  & 188   & 292 & 755
				\\
				\midrule
				\multirow{2}{*}{ Quantile-RaSK\cite{zhang2024quantile} } & IT &10756 & - & -  &-   
				\\
				\cline{2-6}
				& CPU & 0.7984 &  -   & -   &-
				\\
				\midrule
				\multirow{2}{*}{ Quantile-ERaSK\cite{zhang2024quantile} } & IT & 6090 &  - &  - &-   
				\\
				\cline{2-6}
				& CPU & 0.5953 & -  &  -&-
				\\
				\midrule
				\multirow{2}{*}{ Quantile-RaSK-MM, $\gamma=0.01$ } & IT &  4500 & 5963 &   14282 & 16191  
				\\
				\cline{2-6}
				& CPU & \textbf{0.2687} &  \textbf{0.2979} & 0.4287 & 0.9393 
				\\
				\midrule
				\multirow{2}{*}{ Quantile-RaSK-MM, $\gamma=0.1$ } & IT & 6005 & 6070  &  12968  & 15837  
				\\
				\cline{2-6}
				& CPU & 0.3495 &  0.3021   &  \textbf{0.3895} & \textbf{0.9195}
				\\
				\bottomrule
			\end{tabular*}
		\end{minipage}
	\end{center}
\end{table}

\begin{figure}[H]
	\centering
	\subfigure[ash219]{
		\includegraphics[width=0.23\linewidth,height=0.2\textwidth]{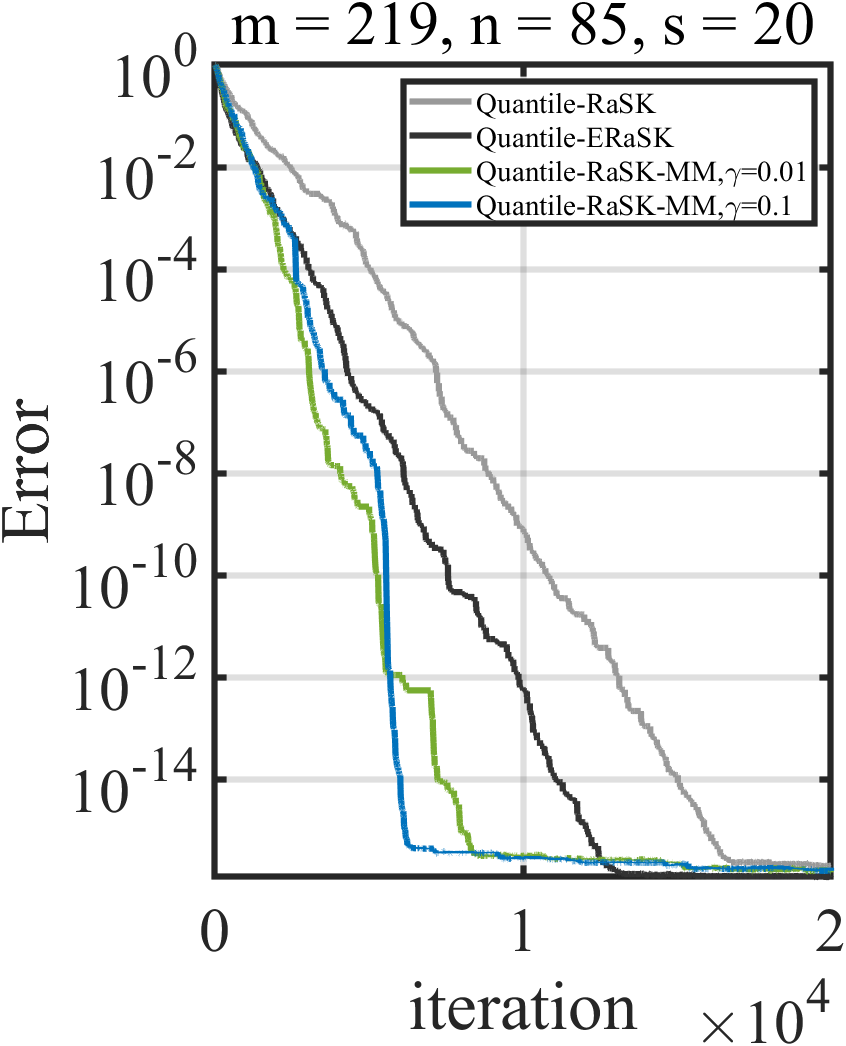}
	\includegraphics[width=0.23\linewidth,height=0.2\textwidth]{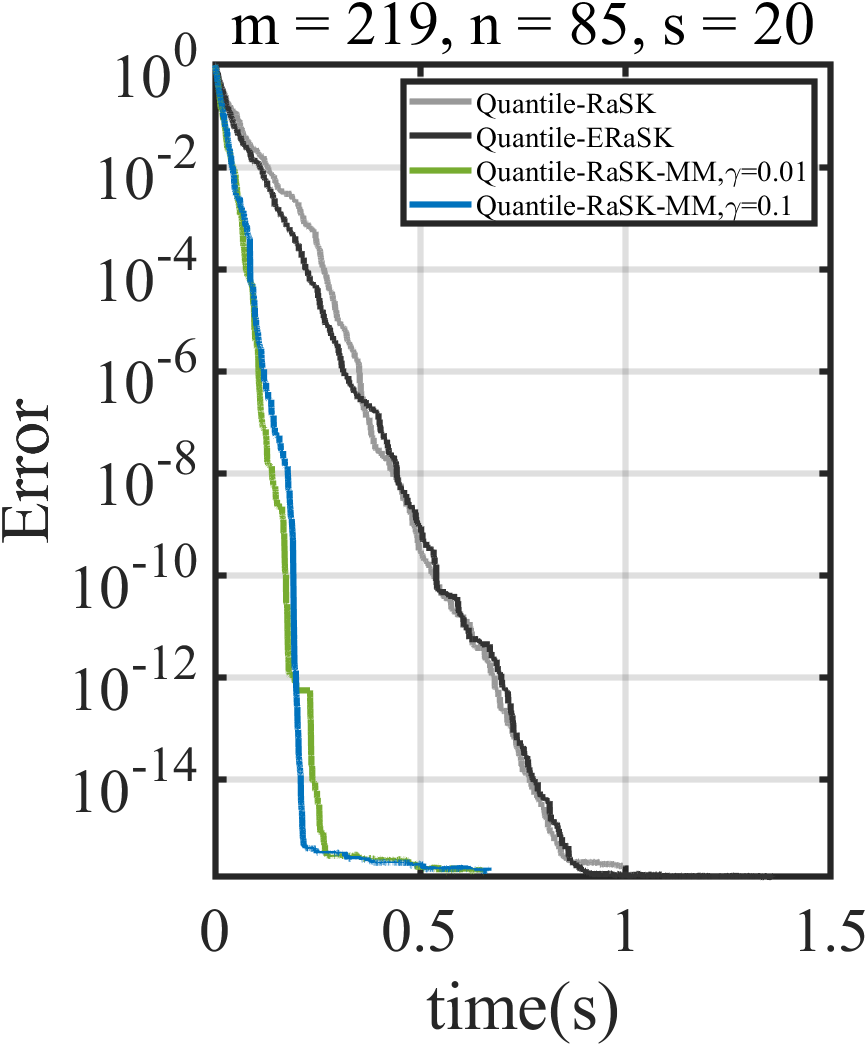}
}
\subfigure[Franz1]{
	\includegraphics[width=0.23\linewidth,height=0.2\textwidth]{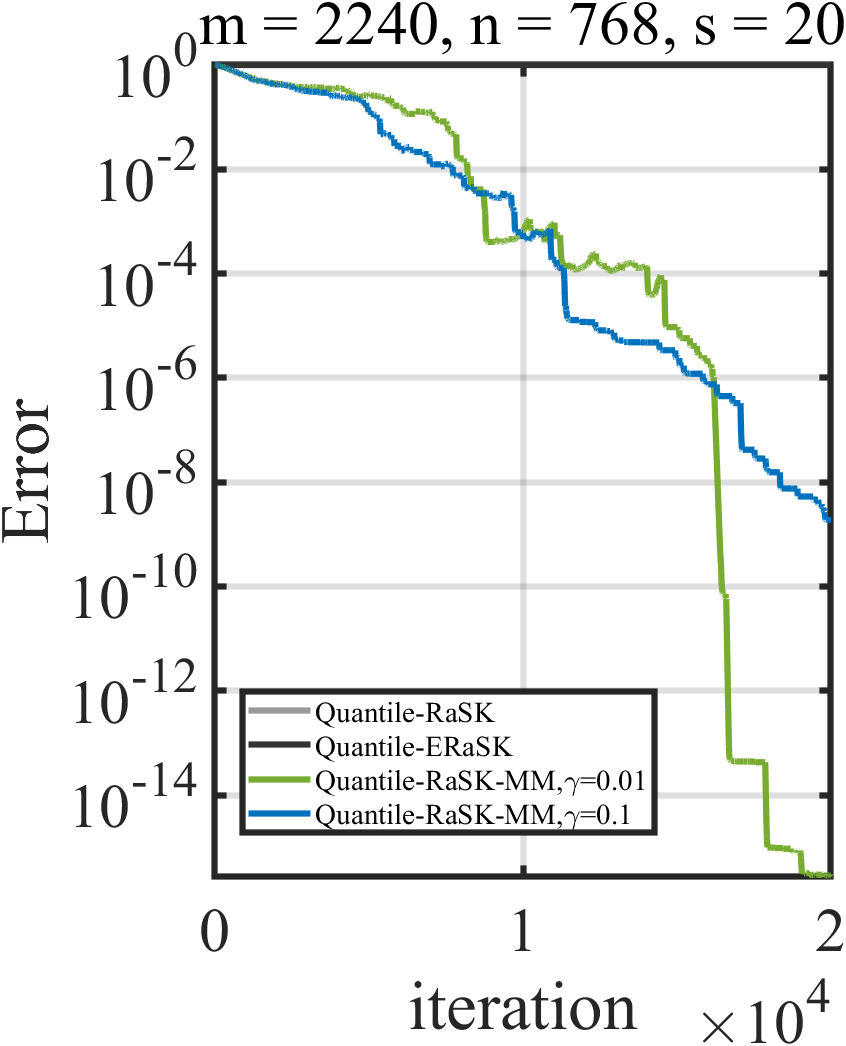}
	\includegraphics[width=0.23\linewidth,height=0.2\textwidth]{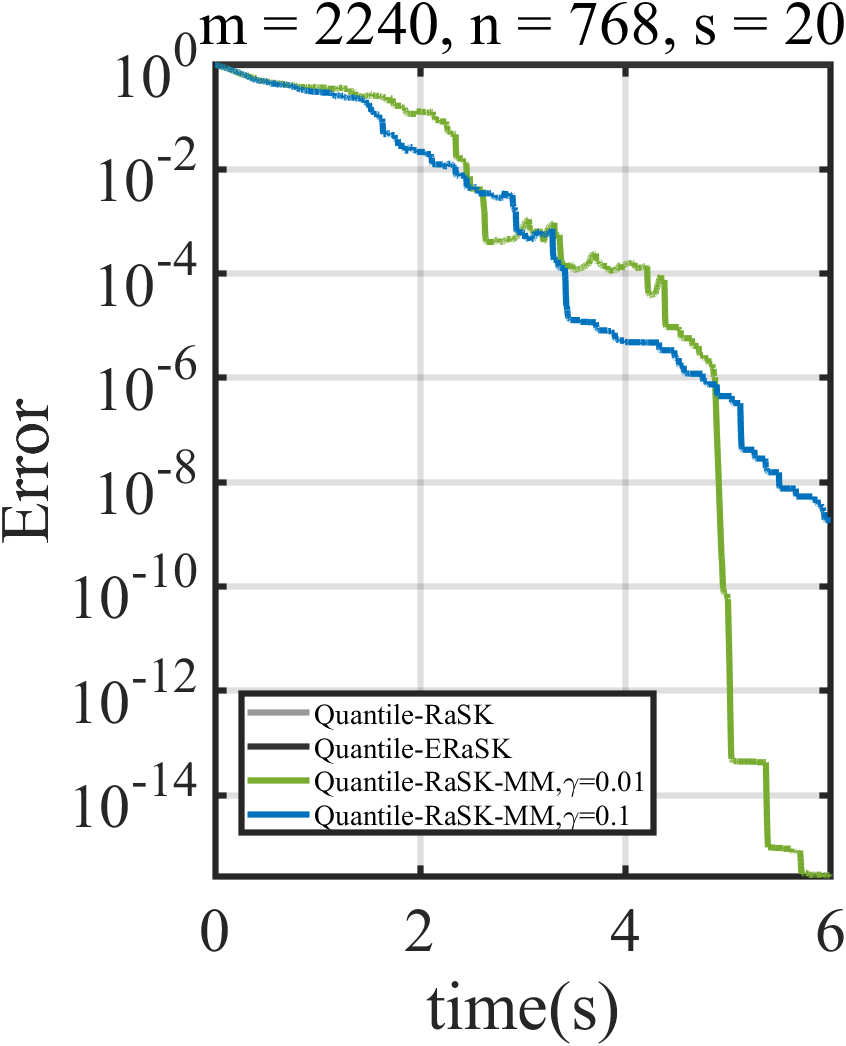}	
}	
	\caption{The performance of different methods on real matrices in corrupted case}
	\label{fig9}
\end{figure}

\subsection{Phantom picture}
\label{sec5.5}

In this subsection, we test the performance of Quantile-RaSK-MM for reconstructing ground-truth images from linear systems arising in computed tomography. Following the parameter setting in Section \ref{sec5:1}, the exact matrix and exact vector are generated by the function \textsl{paralleltomo(N,theta,p,d)} from the AIRTOOLS package in \cite{hansen2012air,hansen2018air}. The reconstruction results for $N=20$ and $N=30$ are summarized in Figure \ref{fig15}.
Both corrupted systems and noisy systems are constructed as described in Section \ref{sec5.1}.
Only the contaminated vector $\tilde{b}$ is available instead of the exact vector $b$, and the task is to reconstruct the ground truth $\hat{x}$ from the inexact system $Ax\approx \tilde{b}$.

Let $\beta=0.2,q=0.8$, and $\gamma=0.01$.
We employ the ME rule as the stopping criterion and output the final iterate. The output is reshaped into a matrix using the function \textsl{reshape}. We evaluate the numerical performance of Quantile-RaSK-MM by the peak signal-to-noise ratio (PSNR) of the recovered images. As shown in Figure \ref{fig15}, Quantile-RaSK-MM equipped with the proposed ME stopping rule yields accurate reconstructions for both corrupted and noisy systems.


\begin{figure}[H]
	\centering
	\subfigure[$N=20$]{
		\includegraphics[width=0.12\linewidth,height=0.16\textwidth,keepaspectratio]{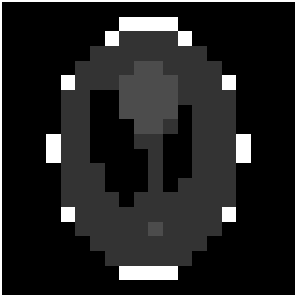}
		\hspace{0.01\linewidth}
		\includegraphics[width=0.15\linewidth,height=0.16\textwidth,keepaspectratio]{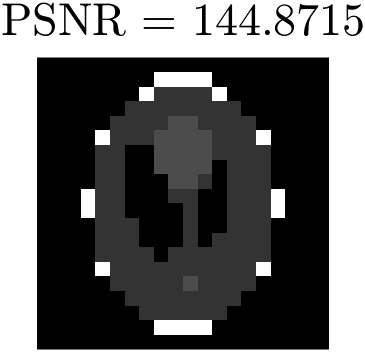}
		\hspace{0.01\linewidth}
		\includegraphics[width=0.14\linewidth,height=0.16\textwidth,keepaspectratio]{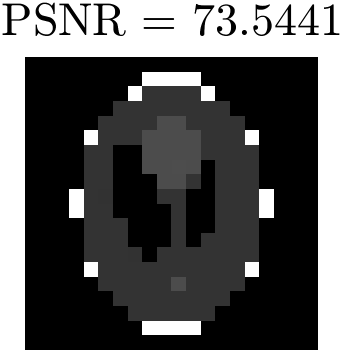}
	}
\hspace{0.05\linewidth}
	\subfigure[$N=30$]{
		\includegraphics[width=0.12\linewidth,height=0.16\textwidth,keepaspectratio]{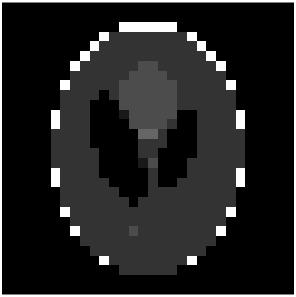}
		\hspace{0.01\linewidth}
		\includegraphics[width=0.14\linewidth,height=0.16\textwidth,keepaspectratio]{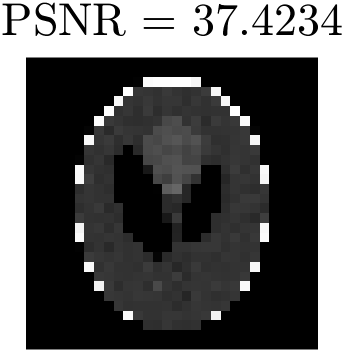}
		\hspace{0.01\linewidth}
		\includegraphics[width=0.14\linewidth,height=0.16\textwidth,keepaspectratio]{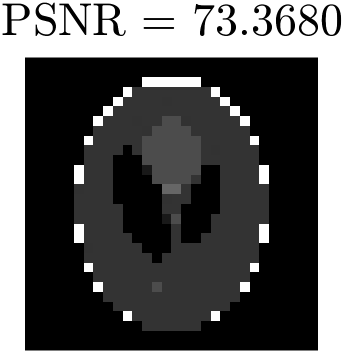}
	}
	\caption{CT reconstructions by Quantile-RaSK-MM with the ME stopping rule. Rows (left-to-right): ground truth for $N=20$, reconstructions from corrupted systems for $N=20$, reconstructions from noisy systems for $N=20$; ground truth for $N=30$, reconstructions from corrupted systems for $N=30$, reconstructions from noisy systems for $N=30$.}
	\label{fig15}
\end{figure}


\section{Conclusion}
\label{sec6}

We proposed the Quantile-based randomized sparse Kaczmarz method with heavy ball momentum method, a novel approach tailored for solving corrupted linear systems.
The relaxed parameter and momentum parameter of this method are meticulously designed based on the minimal dual function principle, addressing the limitation of the minimal-error principle in \cite{lorenz2023minimal,zeng2024adaptiveh} that necessitates the knowledge of the exact right-hand vector. As a comprehensive mathematical framework, the Quantile-RaSK-MM method is versatile enough to be reduced to RaSK-MM applied to noisy or exact systems.
In addition, in order to reduce the semi-convergence phenomenon arising from excessive iterations, we design two stopping criteria for Quantile-RaSK-MM and RaSK-MM based on the DP and ME rule.
Theoretically, we established that the Quantile-RaSK-MM method converges to a corrupted horizon. 
Numerical experiments have not only confirmed the theoretical results but also demonstrated the effectiveness of our proposed methods in recovering solutions of corrupted or noisy linear systems.

In the future work, we would like to develop strategies to eliminate the convergence horizon induced by noise and corruption. It would be a key issue to reduce the influence of noise and corruption on the convergence results of our proposed algorithms. 
Moreover, building on the frameworks in \cite{jin2023stochastic,jinconvergence,jin2024adaptive}, we expect to establish regularization properties and convergence rates for Quantile-RaSK-MM under benchmark source conditions in infinite-dimensional spaces.

\section*{Acknowledgements}

The authors would like to thank the anonymous referees and the associate editor for valuable suggestions and comments, which allowed us to improve the original presentation. 
This work was funded by the National Natural Science Foundation of China (No.12471300, No.61977065), and the China Scholarship Council program (Project ID: 202406110004).
The authors declare no competing interests. All data that support the findings of this study are included within the article. All authors contributed to the study’s conception and design. The first draft of the
manuscript was written by Lu Zhang, and all authors commented on previous versions of the manuscript.
All authors read and approve the final manuscript and are all aware of the current submission.

\bibliographystyle{plain}
\bibliography{ref}

@article{schopfer2019linear,
  title={Linear convergence of the randomized sparse {K}aczmarz method},
  author={Sch{\"o}pfer, Frank and Lorenz, Dirk A},
  journal={Mathematical Programming},
  volume={173},
  pages={509--536},
  year={2019},
  publisher={Springer}
}

@article{jinconvergence,
	title={Convergence rates of a dual gradient method for constrained linear ill-posed problems},
	author={Jin, Qinian},
	journal={Numerische Mathematik},
	volume={151},
	number={4},
	pages={841--871},
	year={2022},
	publisher={Springer}
}

@article{hamarik2001monotone,
	title={On the monotone error rule for parameter choice in iterative and continuous regularization methods},
	author={H{\"a}marik, U and Tautenhahn, U},
	journal={BIT Numerical Mathematics},
	volume={41},
	number={5},
	pages={1029--1038},
	year={2001},
	publisher={Springer}
}

@article{hansen2018air,
	title={AIR Tools {II}: algebraic iterative reconstruction methods, improved implementation},
	author={Hansen, Per Christian and J{\o}rgensen, Jakob Sauer},
	journal={Numerical Algorithms},
	volume={79},
	number={1},
	pages={107--137},
	year={2018},
	publisher={Springer}
}

@article{hansen2012air,
	title={AIR tools—a {MATLAB} package of algebraic iterative reconstruction methods},
	author={Hansen, Per Christian and Saxild-Hansen, Maria},
	journal={Journal of Computational and Applied Mathematics},
	volume={236},
	number={8},
	pages={2167--2178},
	year={2012},
	publisher={Elsevier}
}

@article{lorenz2023minimal,
	title={Minimal error momentum {B}regman-{K}aczmarz},
	author={Lorenz, Dirk A and Winkler, Maximilian},
	journal={Linear Algebra and its Applications},
	volume={709},
	pages={416--448},
	year={2025},
	publisher={Elsevier}
}

@article{davis2011university,
	title={The {U}niversity of Florida sparse matrix collection},
	author={Davis, Timothy A and Hu, Yifan},
	journal={ACM Transactions on Mathematical Software (TOMS)},
	volume={38},
	number={1},
	pages={1--25},
	year={2011},
	publisher={ACM New York, NY, USA}
}

@article{zhang2024quantile,
	title={Quantile-based random sparse {K}aczmarz for corrupted and noisy linear systems},
	author={Zhang, Lu and Wang, Hongxia and Zhang, Hui},
	journal={Numerical Algorithms},
	pages={1--36},
	year={2024},
	publisher={Springer}
}

@article{cai2025subsample,
	title={On Subsample Size of Quantile-Based Randomized {K}aczmarz},
	author={Cai, JianFeng and Chen, Junren and Ma, Anna and Wu, Tong},
	journal={arXiv preprint arXiv:2507.15185},
	year={2025}
}

@article{jin2023stochastic,
	title={Stochastic mirror descent method for linear ill-posed problems in Banach spaces},
	author={Jin, Qinian and Lu, Xiliang and Zhang, Liuying},
	journal={Inverse Problems},
	volume={39},
	number={6},
	pages={065010},
	year={2023},
	publisher={IOP Publishing}
}

@inproceedings{haddock2023subsampled,
	title={On subsampled quantile randomized {K}aczmarz},
	author={Haddock, Jamie and Ma, Anna and Rebrova, Elizaveta},
	booktitle={2023 59th Annual Allerton Conference on Communication, Control, and Computing (Allerton)},
	pages={1--8},
	year={2023},
	organization={IEEE}
}

@article{cheng2022block,
	title={On block accelerations of quantile randomized {K}aczmarz for corrupted systems of linear equations},
	author={Cheng, Lu and Jarman, Benjamin and Needell, Deanna and Rebrova, Elizaveta},
	journal={Inverse Problems},
	volume={39},
	number={2},
	pages={024002},
	year={2022},
	publisher={IOP Publishing}
}

@article{tondji2024acceleration,
	title={Acceleration and restart for the randomized {B}regman-{K}aczmarz method},
	author={Tondji, Lionel and Necoara, Ion and Lorenz, Dirk A},
	journal={Linear Algebra and its Applications},
	volume={699},
	pages={508--538},
	year={2024},
	publisher={Elsevier}
}

@book{2017Convex,
	title={Convex Analysis and Monotone Operator Theory in Hilbert Spaces},
	author={Heinz H.Bauschke, Patrick L.Combettes},
	publisher={Springer, Cham, second edition},
	year={2017},
}

@article{petra2015randomized,
	title={Randomized sparse block {K}aczmarz as randomized dual block-coordinate descent},
	author={Petra, Stefania},
	journal={Analele {\c{s}}tiin{\c{t}}ifice ale Universit{\u{a}}{\c{t}}ii" Ovidius" Constan{\c{t}}a. Seria Matematic{\u{a}}},
	volume={23},
	number={3},
	pages={129--149},
	year={2015}
}

@article{zeng2024adaptive,
	title={On adaptive stochastic extended iterative methods for solving least squares},
	author={Zeng, Yun and Han, Deren and Su, Yansheng and Xie, Jiaxin},
	journal={Mathematics of Computation},
	year={2025}
}

@article{jin2024adaptive,
	title={An adaptive heavy ball method for ill-posed inverse problems},
	author={Jin, Qinian and Huang, Qin},
	journal={SIAM Journal on Imaging Sciences},
	volume={17},
	number={4},
	pages={2212--2241},
	year={2024},
	publisher={SIAM}
}

@article{jin2024convergence,
	title={Convergence analysis of a stochastic heavy-ball method for linear ill-posed problems},
	author={Jin, Qinian and Liu, Yanjun},
	journal={Journal of Computational and Applied Mathematics},
	pages={116702},
	year={2025},
	publisher={Elsevier}
}

@article{nesterov2012efficiency,
	title={Efficiency of coordinate descent methods on huge-scale optimization problems},
	author={Nesterov, Yu},
	journal={SIAM Journal on Optimization},
	volume={22},
	number={2},
	pages={341--362},
	year={2012},
	publisher={SIAM}
}

@article{cai2009convergence,
	title={Convergence of the linearized {B}regman iteration for $l_1$-norm minimization},
	author={Cai, Jian Feng and Osher, Stanley and Shen, Zuowei},
	journal={Mathematics of Computation},
	volume={78},
	number={268},
	pages={2127--2136},
	year={2009}
}

@inproceedings{jarman2021quantilerk,
	title={Quantile{RK}: Solving large-scale linear systems with corrupted, noisy data},
	author={Jarman, Benjamin and Needell, Deanna},
	booktitle={2021 55th Asilomar Conference on Signals, Systems, and Computers},
	pages={1312--1316},
	year={2021},
	organization={IEEE}
}

@article{zhang2022weighted,
	title={A weighted randomized sparse {K}aczmarz method for solving linear systems},
	author={Zhang, Lu and Yuan, Ziyang and Wang, Hongxia and Zhang, Hui},
	journal={Computational and Applied Mathematics},
	volume={41},
	number={8},
	pages={383},
	year={2022},
	publisher={Springer}
}

@article{elfving2007stopping,
	title={Stopping rules for {L}andweber-type iteration},
	author={Elfving, Tommy and Nikazad, Touraj},
	journal={Inverse Problems},
	volume={23},
	number={4},
	pages={1417},
	year={2007},
	publisher={IOP Publishing}
}

@article{necoara2019faster,
	title={Faster randomized block {K}aczmarz algorithms},
	author={Necoara, Ion},
	journal={SIAM Journal on Matrix Analysis and Applications},
	volume={40},
	number={4},
	pages={1425--1452},
	year={2019},
	publisher={SIAM}
}

@article{tondji2023faster,
	title={Faster randomized block sparse {K}aczmarz by averaging},
	author={Tondji, Lionel and Lorenz, Dirk A},
	journal={Numerical Algorithms},
	volume={93},
	number={4},
	pages={1417--1451},
	year={2023},
	publisher={Springer}
}

@article{yuan2022sparse,
	title={Sparse sampling {K}aczmarz--{M}otzkin method with linear convergence},
	author={Yuan, Ziyang and Zhang, Hui and Wang, Hongxia},
	journal={Mathematical Methods in the Applied Sciences},
	volume={45},
	number={7},
	pages={3463--3478},
	year={2022},
	publisher={Wiley Online Library}
}

@article{tondji2024adaptive,
	title={Adaptive {B}regman--{K}aczmarz: an approach to solve linear inverse problems with independent noise exactly},
	author={Tondji, Lionel and Tondji, Idriss and Lorenz, Dirk},
	journal={Inverse Problems},
	volume={40},
	number={9},
	pages={095006},
	year={2024},
	publisher={IOP Publishing}
}

@article{zeng2024adaptiveh,
	title={On adaptive stochastic heavy ball momentum for solving linear systems},
	author={Zeng, Yun and Han, Deren and Su, Yansheng and Xie, Jiaxin},
	journal={SIAM Journal on Matrix Analysis and Applications},
	volume={45},
	number={3},
	pages={1259--1286},
	year={2024},
	publisher={SIAM}
}

@inproceedings{ghadimi2015global,
	title={Global convergence of the heavy-ball method for convex optimization},
	author={Ghadimi, Euhanna and Feyzmahdavian, Hamid Reza and Johansson, Mikael},
	booktitle={2015 European control conference (ECC)},
	pages={310--315},
	year={2015},
	organization={IEEE}
}

@article{polyak1964some,
	title={Some methods of speeding up the convergence of iteration methods},
	author={Polyak, Boris T},
	journal={Ussr computational mathematics and mathematical physics},
	volume={4},
	number={5},
	pages={1--17},
	year={1964},
	publisher={Elsevier}
}

@article{loizou2020momentum,
	title={Momentum and stochastic momentum for stochastic gradient, newton, proximal point and subspace descent methods},
	author={Loizou, Nicolas and Richt{\'a}rik, Peter},
	journal={Computational Optimization and Applications},
	volume={77},
	number={3},
	pages={653--710},
	year={2020},
	publisher={Springer}
}

@article{haddock2022quantile,
	title={Quantile-based iterative methods for corrupted systems of linear equations},
	author={Haddock, Jamie and Needell, Deanna and Rebrova, Elizaveta and Swartworth, William},
	journal={SIAM Journal on Matrix Analysis and Applications},
	volume={43},
	number={2},
	pages={605--637},
	year={2022},
	publisher={SIAM}
}

@article{yin2010analysis,
	title={Analysis and generalizations of the linearized {B}regman method},
	author={Yin, Wotao},
	journal={SIAM Journal on Imaging Sciences},
	volume={3},
	number={4},
	pages={856--877},
	year={2010},
	publisher={SIAM}
}

@article{lorenz2014linearized,
	title={The linearized {B}regman method via split feasibility problems: analysis and generalizations},
	author={Lorenz, Dirk A and Schopfer, Frank and Wenger, Stephan},
	journal={SIAM Journal on Imaging Sciences},
	volume={7},
	number={2},
	pages={1237--1262},
	year={2014},
	publisher={SIAM}
}

@article{karczmarz1937angenaherte,
	title={Angenaherte auflosung von systemen linearer glei-chungen},
	author={Karczmarz, Stefan},
	journal={Bull. Int. Acad. Pol. Sic. Let., Cl. Sci. Math. Nat.},
	pages={355--357},
	year={1937}
}

@article{strohmer2009randomized,
	title={A randomized {K}aczmarz algorithm with exponential convergence},
	author={Strohmer, Thomas and Vershynin, Roman},
	journal={Journal of Fourier Analysis and Applications},
	volume={15},
	number={2},
	pages={262--278},
	year={2009},
	publisher={Springer}
}

@article{needell2010randomized,
	title={Randomized {K}aczmarz solver for noisy linear systems},
	author={Needell, Deanna},
	journal={BIT Numerical Mathematics},
	volume={50},
	pages={395--403},
	year={2010},
	publisher={Springer}
}

@article{zouzias2013rek,
	title={Randomized Extended {K}aczmarz for Solving Least Squares},
	author={Zouzias, Anastasios and Freris, Nikolaos M.},
	journal={SIAM Journal on Matrix Analysis and Applications},
	volume={34},
	number={2},
	pages={773--793},
	year={2013},
	publisher={SIAM},
	doi={10.1137/120889897}
}

@article{steinerberger2021sv,
	title={Randomized {K}aczmarz Converges Along Small Singular Vectors},
	author={Steinerberger, Stefan},
	journal={SIAM Journal on Matrix Analysis and Applications},
	volume={42},
	number={2},
	pages={608--615},
	year={2021},
	publisher={SIAM},
	doi={10.1137/20M1350947}
}

@article{steinerberger2023quantile,
	title={Quantile-based random {K}aczmarz for corrupted linear systems of equations},
	author={Steinerberger, Stefan},
	journal={Information and Inference: A Journal of the IMA},
	volume={12},
	number={1},
	pages={448--465},
	year={2023},
	publisher={Oxford University Press}
}

@article{coria2024quantile,
	title={On Quantile Randomized {K}aczmarz for Linear Systems with Time-Varying Noise and Corruption},
	author={Coria, Nestor and Haddock, Jamie and Pacheco, Jaime},
	journal={arXiv preprint arXiv:2403.19874},
	year={2024}
}

@book{hansen2021computed,
	title={Computed tomography: algorithms, insight, and just enough theory},
	author={Hansen, Per Christian and J{\o}rgensen, Jakob and Lionheart, William RB},
	year={2021},
	publisher={SIAM}
}

@article{elfving2014semi,
	title={Semi-convergence properties of {K}aczmarz’s method},
	author={Elfving, Tommy and Hansen, Per Christian and Nikazad, Touraj},
	journal={Inverse problems},
	volume={30},
	number={5},
	pages={055007},
	year={2014},
	publisher={IOP Publishing}
}

@article{bottou2018optimization,
	title={Optimization methods for large-scale machine learning},
	author={Bottou, L{\'e}on and Curtis, Frank E and Nocedal, Jorge},
	journal={SIAM review},
	volume={60},
	number={2},
	pages={223--311},
	year={2018},
	publisher={SIAM}
}

@book{hastie2009elements,
	title={The elements of statistical learning: data mining, inference, and prediction},
	author={Hastie, Trevor and Tibshirani, Robert and Friedman, Jerome H and Friedman, Jerome H},
	volume={2},
	year={2009},
	publisher={Springer}
}

@article{yuan2022adaptively,
	title={Adaptively sketched {B}regman projection methods for linear systems},
	author={Yuan, Ziyang and Zhang, Lu and Wang, Hongxia and Zhang, Hui},
	journal={Inverse Problems},
	volume={38},
	number={6},
	pages={065005},
	year={2022},
	publisher={IOP Publishing}
}

@article{battaglia2024reverse,
	title={Reverse Quantile-RK and Its Application to Quantile-RK},
	author={Battaglia, Emeric and Ma, Anna},
	journal={Numerical Linear Algebra with Applications},
	volume={32},
	number={3},
	pages={e70024},
	year={2025},
	publisher={Wiley Online Library}
}

\end{document}